\documentclass{amsart}
\usepackage{amssymb}
\newcommand{\e}{\epsilon}

\newcommand{\abs}[1]{\lvert#1\rvert}
\newcommand{\babs}[1]{\big\lvert#1\big\rvert}

\newcommand{\dsp}{\displaystyle}
\newcommand{\dt}{\partial_t}

\newcommand{\R}{{\mathbb R}}
\newcommand{\N}{{\mathbb N}}
\newcommand{\eps}{\epsilon}

\newcommand{\bv}{{\mathbf v}}

\newcommand{\Me}{M_\epsilon(\phi)}
\newcommand{\Meb}{M_\epsilon(\underline{\phi})}

\newcommand{\Neab}{N_{\epsilon,\alpha}(\underline{\phi},\underline{n})}
\newcommand{\Nea}{N_{\epsilon,\alpha}({\phi},{n})}
\newcommand{\un}{\underline{n}}
\newcommand{\ubv}{\underline{\bv}}

\theoremstyle{remark}

\newtheorem{remark}{Remark}
\newtheorem{notation}{Notation}
\theoremstyle{definition}

\theoremstyle{theorem}
\newtheorem{proposition}{Proposition}
\newtheorem{lemma}{Lemma}
\newtheorem{theorem}{Theorem}

\newtheorem*{merci}{Acknowledgements}

\begin{document}
\title[ZK equation]{The Cauchy problem for the Euler-Poisson system and  derivation of the Zakharov-Kuznetsov equation}
\author{David Lannes}
\address{ D.L : DMA, ENS Ulm,\\ 45 rue d'Ulm, 75230 Paris, France\\ E-mail lannes@ens.fr}
\author{Felipe Linares}
\address{ F.L : IMPA,\\ Estrada Dona Castorina 110\\ Rio de Janeiro 22460-320, RJ Brasil\\ E-mail linares@impa.br}
\author{Jean-Claude Saut}
\address{ J.-C.S.: Laboratoire de Math\' ematiques, UMR 8628,\\
Universit\' e Paris-Sud et CNRS,\\ 91405 Orsay, France\\ E-mail jean-claude.saut@math.u-psud.fr}
\date{April 19, 2012}
\maketitle
\begin{abstract}
%\textit{Abstract}
We consider in this paper the rigorous justification of the
Zakharov-Kuznetsov equation from the Euler-Poisson system for
uniformly magnetized plasmas. We first provide a proof of the local
well-posedness of the Cauchy problem for the aforementioned system in
dimensions two and three. Then we prove that  the long-wave
small-amplitude limit is described by the Zakharov-Kuznetsov
equation. This is done first in the case of cold plasma; we then show
how to extend this result in presence of the isothermal pressure term
with uniform estimates when this latter goes to zero.
\end{abstract}

\section{Introduction}

\subsection{General Setting}

The Zakharov-Kuznetsov equation
\begin{equation}\label{ZK}
 u_t+u\,\partial_xu+\partial_x\Delta u=0, \qquad u=u(x,y,z,t), \qquad (x,y,z)\in \R^d,\, t\in \R,\, d=2, 3
  \end{equation}
 was introduced as an asymptotic model in \cite{ZK} (see also \cite{LaS}, \cite{IR}, \cite{D}, and  \cite{VMPH} for some generalizations) to describe the propagation of nonlinear ionic-sonic waves in a magnetized plasma.  
  
   The Zakharov-Kuznetsov is  a natural multi-dimensional extension of the Korteweg-de Vries equation,
   quite different from the well-known Kadomtsev-Petviashvili (KP) equation though.

Contrary to the Korteweg-de Vries or the Kadomtsev-Petviashvili
equations, the Zakharov-Kuznetsov equation is not completely integrable but
it has a hamiltonian structure and possesses two invariants, namely
(for $u_0=u(\cdot,0)$) :
\begin{equation}\label{L^2}
M(t) = \int_{R^d} u^2(x,t)=\int_{R^d} u_0^2(x) = M(0)
\end{equation}
and the hamiltonian
\begin{equation}\label{H}
H(t) = \frac{1}{2}\int_{R^d}[  |\nabla u|^2 -\frac{u^3}{3}]= \frac{1}{2}\int_{R^d}[ |\nabla u_0|^2 -\frac{u_0^3}{3}]=H(0).
\end{equation}
The Cauchy problem for the Zakharov-Kuznetsov equation has been proven to be globally well posed in the two-dimensional case for data in $H^1(\R^2)$ (\cite{Fa}), and locally well-posed in the three-dimensional case for data in $H^s(\R^3)$, $s>\frac{3}{2}$, (\cite{LiS}) and recently in $H^s(\R^3),$ $ s>1,$ \cite{RV}. We also refer to  \cite{LPS} for solutions on a nontrivial background and to \cite{LiPa} and  \cite{RV2} for well-posedness results of the Cauchy problem for {\it generalized} Zakharov-Kuznetsov equations in $\R^2.$ Unique continuation properties for the Zakharov-Kuznetsov equation were established in \cite{Pa}, \cite{BIM}.

The Zakharov-Kuznetsov equation was formally derived in \cite{ZK}  as
a long wave small-amplitude limit of the following Euler-Poisson
system in the ``cold plasma'' approximation,
\begin{equation}
    \label{eulerpoisson}
    \left\lbrace
   \begin{array}{lcl}
   	\dsp\partial_tn+\nabla \cdot {\bf v}+\nabla\cdot(n{\bf v})=0,\\
	\dsp\partial_t{\bf v}+({\bf v}\cdot\nabla){\bf v}+\nabla\phi+a{\bf e}\wedge {\bf v}=0,\\
	\dsp\Delta\phi-e^{\phi}+1+n=0.
    \end{array}\right.
\end{equation}
Here $n$ is the deviation of the ion density from $1$, ${\bf v}$ is the ion velocity, $\phi$ the electric potential, $a$ is a measure of  the uniform magnetic field, applied along the vector ${\bf e}=(1,0,0)^T$ so that if ${\bf v}=(v_1,v_2,v_3)^T,$ ${\bf e}\wedge {\bf v}=(0,-v_3,v_2)^T.$ \\
Note that this skew-adjoint term is similar to a Coriolis term in the Euler equations for inviscid incompressible fluids.

The main goal of the present paper is to justify rigorously this
formal long-wave limit. The one-dimensional case (leading to the
Korteweg-de Vries equation) has been partially justified in
\cite{C-BY}, and Guo-Pu gave in a recent preprint a full justification
of this limit \cite{GuoPu}.

Setting $\rho = (1+n)$, (\ref{eulerpoisson}) writes 
\begin{equation}\label{eulerpoissonbis}
  \left\lbrace
   \begin{array}{lcl}
   	\dsp\partial_t\rho+\nabla\cdot(\rho {\bf v})=0,\\
	\dsp\partial_t{\bf v}+({\bf v}\cdot\nabla){\bf v}+\nabla\phi+a{\bf e}\wedge {\bf v}=0,\\
	\dsp\Delta\phi-e^{\phi}+\rho=0.
    \end{array}\right.
\end{equation}
In this formulation, the Euler-Poisson system possesses a (formally) hamiltonian conserved energy
\begin{eqnarray}
\nonumber
H(\rho,{\bf v}, \phi) &=& \int_{\R^d}\lbrack \frac{1}{2}\rho|{\bf
  v}|^2+\rho\phi-\frac{1}{2}|\nabla \phi|^2-e^{\phi}+1\rbrack dx\\
\label{hamiltonian}
&=& \int_{\R^d}\lbrack \frac{1}{2}\rho|{\bf v}|^2+\frac{1}{2}|\nabla \phi|^2+e^{\phi}(\phi -1)+1\rbrack dx,
\end{eqnarray}
as it is easily seen by multiplying the first (resp. second) equation in \eqref{eulerpoissonbis} by $\frac{1}{2}|{\bf v}|^2+\phi$ (resp. $\rho {\bf v}$), then adding and integrating over $\R^d$.\\
Of course one has to find a correct functional setting in order to justify the definition of $H$ and its conservation (see Remark  1 in Subsection  2.1 below).\\
Note that the Euler-Poisson system with $a=0$ has another formally conserved quantity, namely the impulse
\begin{equation}
\label{impulse}
\frac{d}{dt}\int_{\R^d}\rho{\bf v}dx =0.
\end{equation}
This conservation law is easily derived (formally) from the equation for $\rho {\bf v}$:
\begin{equation}\label{eqrhov}
\partial_t(\rho{\bf v})+\nabla\cdot(\rho{\bf v}\otimes{\bf v})+\rho\nabla \phi=0.
\end{equation}

Linearizing (\ref{eulerpoisson}) around the constant solution $n=1,$ ${\bf v} =0,$ $\phi =0,$ one finds the dispersion relation for a plane wave $e^{i(\omega t-{\bf k}. x)}$,  $ {\bf k} = (k_1, k_2, k_3)$:
\begin{equation}
\label{dispersion}
\omega^4({\bf k}) +\omega^2( {\bf k})(a^2 -\frac{| {\bf k}|^2}{1+|{\bf k}|^2})-a^2\frac{k_1^2}{1+|{\bf k}|^2}=0
\end{equation}
or
\begin{equation}
\label{dispersionbis}
(1+ |{\bf k}|^2) - \frac{k_1^2}{\omega^2({\bf k})} - \frac{|{\bf k}_{\perp}|^2}{\omega^2({\bf k}) + a^2} =0
\end{equation}
where $|{\bf k}_{\perp}|^2 = k_2^2 +  k_3^2$.\\
In the absence of applied magnetic field (a =0), this relation reduces to:
\begin{equation}
\label{a=0}
\omega^2({\bf k})=\frac{|{\bf k}|^2}{1+|{\bf k}|^2}.
\end{equation}
Those relations display the weakly dispersive character of the Euler-Poisson system.

Contrary to the KP case (see \cite{AL} for the justification of various asymptotic models of surface waves) the rigorous justification of the long-wave limit of the Euler-Poisson system has not been carried out so far (see however \cite{C-BY,GuoPu} in the one-dimensional case). This justification is the main goal of the present paper. 

To start with, we investigate the local well-posedness of the Euler-Poisson
system (\ref {eulerpoisson})  which does not raise any particular
difficulty but for which, to our knowledge, no result
seems to be explicitly available in the literature in dimensions $2$
and $3$. The paper \cite{T} concerns a ``linearized" version of (\ref
{eulerpoisson}), namely the term $\Delta\phi +1-e^{\phi}$ is replaced
by its linearization at $\phi =0$ that is $(\Delta -1)\phi$ (see also
\cite{HMT} where a unique continuation property is established for the
one- dimensional version of this ``modified" Euler-Poisson
system). Note that this ``linearized" version of (\ref {eulerpoisson})
is somewhat reminiscent of the pressureless Euler-Poisson system which
has been intensively studied (see for instance  \cite {BC-T},
\cite{DC-T} and the references therein).

In the one-dimensional case  (\ref{eulerpoisson}) has the very simple form
\begin{equation}
\label {oneD}
 \left\lbrace
   \begin{array}{lcl}
   	\dsp\partial_tn+\lbrack(1+n){\bf v}\rbrack_x=0,\\
	\dsp\partial_t{\bf v}+{\bf v}{\bf v}_x+\phi_x=0,\\
	\dsp \phi_{xx}-e^{\phi}+1+n=0.
    \end{array}\right.
\end{equation}
The existence of supersonic solitary waves for  (\ref{oneD}) has been proven in \cite{LS}. The linear stability of those solitary waves was investigated in \cite{HS} and their interactions was studied in \cite{HNS}, in particular in comparison with their approximations in the long wave limit by KdV solitary waves.

Though our analysis is mainly concerned with the higher dimensional
case $d=2,3$, our results apply as well to the system (\ref{oneD}) and
to provide an alternative proof to \cite{GuoPu} of the justification
of the KdV approximation.

Let us mention finally, that  \eqref{eulerpoisson} is valid for cold
plasmas only; in the general case, an isothermal pressure term must be added,
\begin{equation}
    \label{eulerpoissonHOT}
    \left\lbrace
   \begin{array}{lcl}
   	\dsp\partial_tn+\nabla \cdot {\bf v}+\nabla\cdot(n{\bf v})=0,\\
	\dsp\partial_t{\bf v}+({\bf v}\cdot\nabla){\bf
          v}+\nabla\phi+\alpha \frac{\nabla n}{1+n}+a{\bf e}\wedge {\bf v}=0,\\
	\dsp\Delta\phi-e^{\phi}+1+n=0,
    \end{array}\right.
\end{equation}
with $\alpha\geq 0$. For this system (with $\alpha>0$), global existence for small data
has been proved \cite{GuoPausader} in dimension $d=3$ in absence of
magnetic field ($a=0$) and in the irrrotational case.  Still for
$\alpha>0$, in \cite{GVHKR}, the authors
provide for the full equations uniform energy estimate in the
quasineutral limit (i.e. $e^\phi=1+n$) for well prepared initial data
(note that they also handle the initial boundary value problem). The derivation
and justification of the KdV approximation is generalized in
\cite{GuoPu} using a different proof as in the case $\alpha=0$. We
show how to extend our results on the ZK approximation to
(\ref{eulerpoissonHOT}) and provide uniform estimates with respect to
$\alpha$ that allow one to handle the convergence of solutions of
(\ref{eulerpoissonHOT}) to solutions of (\ref{eulerpoisson}) when $\alpha\to 0$.

\subsection{Organization of the paper}
The paper is organized as follows. \\
 In Section \ref{CPEP} we prove that the Cauchy problem for the Euler-Poisson system (\ref{eulerpoisson}) is locally well-posed. The main step is to  express $\phi$ as a function of $n$ by solving the elliptic equation in
 (\ref{eulerpoisson}) by the super and sub-solutions method. This step is of course trivial when one considers the ``linearized" Euler-Poisson system. Then we establish the local well-posedness for  data $(n_0, {\bf v}_0)$ in $H^{s-1}(\R^d)\times H^s(\R^d)^d,\; s>\frac{d}{2}+1$ such that $|n_0|_{\infty}<1,$ thus generalizing the result of \cite{T}. We use in a crucial way  the smoothing property of the map $n\mapsto \phi$.

  In Section \ref{LWEP} we derive rigorously the Zakharov-Kuznetsov equation as a long wave limit of the Euler-Poisson system.
  In order to do this, we need to introduce a small parameter
  $\epsilon$ and to establish for a scaled version of  the
  Euler-Poisson system existence and bounds on the correct time
  scale. However the elliptic equation for $\phi$ provides a smoothing
  effect which is not uniform with respect to $\epsilon$ and this
  makes the Cauchy problem more delicate (we cannot apply the previous
  strategy which would give an existence time shrinking to zero with
  $\epsilon$).  We are thus led to view the Euler-Poisson system as a
  semilinear perturbation of a symmetrizable quasilinear system and we
  have to find the correct symmetrizer.  We obtain in this framework
  well-posedness (with uniform bounds on the correct time scale) for
  data in  $H^{s}(\R^d)\times H^{s+1}_\eps(\R^d)^d,\; s>\frac{d}{2}+1$
  (we refer to \eqref{defHseps} for the definition of $ H^{s+1}_\eps(\R^d)$).\\
 We then prove that this solution is well approximated by the solution of the ZK equation on the relevant time scales.
 
We finally show in Section \ref{ISP} how to modify the results of
Section \ref{LWEP} when the isothermal pressure is not neglected,
i.e. when one works with (\ref{eulerpoissonHOT}) instead of
(\ref{eulerpoisson}). The presence of a nonzero coefficient $\alpha>0$
in front of the isothermal pressure term induces a smoothing effect on
the variable $n$; however, this smoothing effect vanishes as
$\alpha\to 0$, and, in order to obtain an existence time uniform with
respect to $\eps$ and $\alpha$, it is necessary to work in a Banach
scale indexed by these parameters and that is adapted to measure this
smoothing effect.
 
\subsection{Notations}

- Partial differentiation are denoted by subscripts,  $\partial_x$,  $\partial_t$, $\partial_j=\partial_{x_j}$ etc.

- We denote by $\abs{\cdot}_p$ ($1\leq p\leq\infty$) the standard norm of the Lebesgue
spaces $L^p(\R^d)$ ($d=2,3$).

- We use the Fourier multiplier notation: $f(D)u$ is defined as ${\mathcal F}(f(D)u)(\xi)=f(\xi)
\widehat{u}(\xi)$, where ${\mathcal F}$ and $\widehat{\cdot}$ stand for the Fourier transform.

- The operator $\Lambda=(1-\Delta)^{1/2}$ is equivalently defined using the Fourier
multiplier notation to be $\Lambda=(1+\vert D\vert^2)^{1/2}$.

- The standard notation $H^s(\R^d)$, or simply $H^s$ if the underlying
domain is clear from the context, is used for the $L^2$-based
Sobolev spaces; their norm is written $\vert \cdot\vert_{H^s}$.

- For a given Banach space $X$ we will denote $|.|_{X,T}$ the norm in $C(\lbrack 0,T\rbrack; X).$ When 
$X=L^p$, the corresponding norm will be denoted $|.|_{p,T}$.

- We will denote by $C$ various absolute constants.

- The notation $A+\langle B\rangle_{s>t_0}$ refers to $A$ if $s\leq
t_0$ and $A+B$ if $s>t_0$.
\section{The Cauchy problem for the Euler-Poisson system}\label{CPEP}

The aim of this section is to prove the local well-posedness of the Cauchy problem associated to the Euler-Poisson system.
\begin{equation}
    \label{ivp-eulerpoisson}
    \left\lbrace
   \begin{array}{lcl}
   	\dsp\partial_tn+\nabla \cdot {\bf v}+\nabla\cdot(n{\bf v})=0,\\
	\dsp\partial_t{\bf v}+({\bf v}\cdot\nabla){\bf v}+\nabla\phi+a{\bf e}\wedge {\bf v}=0,\\
	\dsp\Delta\phi-e^{\phi}+1+n=0\\
	n(\cdot,0)=n_0,\\
	{\bf v}(\cdot,0)={\bf v}_0.
    \end{array}\right.
\end{equation}

\subsection{Solving the elliptic part}

We consider here, for a given $n$, the elliptic equation
\begin{equation}\label{elliptic}
L(\phi)=-\Delta \phi +e^{\phi}-1=n.
\end{equation}

\begin{proposition}\label{prop1}
Let $n \in L^{\infty}(\R^d)\cap H^1(\R^d)$, $d=1,2,3$,  such that $\inf_{\R^d}1+n>0$. Then there exists a unique solution $\phi \in 
H^3(\R^d)$ of
(\ref{elliptic}) such that:\\
(i) The following estimate holds
$$
 -K_-=\ln (1-|n|_{\infty})\leq \phi \leq K_+=\ln(1+|n|_\infty).
$$
(ii) Defining $c_\infty(n)=\abs{ (1+n)^{-1}}_\infty$ and $I_1(n)= \frac{1}{c_\infty(n)}\abs{n}_2^2+\abs{n}_{H^1}^2$, one has
$$
\int_{\R^d}\lbrack  \frac{c_{\infty}(n)}{2}|\phi|^2  +\frac{1}{2}|\nabla \phi|^2 +|\Delta \phi|^2 +2e^{\phi}|\nabla \phi|^2 +\frac{1}{2}(e^{\phi}-1)^2\rbrack 
\leq \frac{1}{2}I_1(n).
$$
(iii) If furthermore, $n \in H^s(\R^d),$\, $s\geq 0$, then $\phi \in H^{s+2}(\R^d$) and
\begin{equation}\label{higher}
|\Lambda^{s+1}\nabla \phi|_2\leq  F_s(I_1(n),\abs{n}_\infty)(1+\abs{n}_{H^s}\big),
\end{equation}
where $F_s(.,.)$ is an increasing function of its arguments.
\end{proposition}

\begin{remark}\label{triv}
Writing $\xi_i\xi_j\hat{\phi } = -\frac{\xi_i \xi_j}{|\xi|^2} |\xi|^2 \hat{\phi } $ and using the $L^2$ continuity of the Riesz transforms we see that we can replace $\Delta \phi$ in the left hand side of (ii)  in Proposition \ref{prop1} by any derivative $\partial^{\alpha}\phi,$ replacing possibly the right hand side by $C I_1(n),$ where $C$ is an absolute constant.
\end {remark}

\begin{remark}\label{triv2}
Notice that by (ii), one has $|\phi|^2_2\leq \frac{I_1(n)}{c_{\infty}(n)}.$
\end{remark}

\begin{proof}
We use the method of sub and super-solutions to construct the solution $\phi$. As a super-solution we take $\phi_+ =K_+$ where $K_+$ is a positive constant satisfying 
$$K_+\geq \ln(1+|n|_{\infty}),$$
so that 
$$L(\phi_+)\geq n.$$
As a sub-solution, we choose $\phi_-=-K_-<0$ where $K_-$ is a positive constant satisfying
$$K_-\geq \ln(\frac{1}{\inf_{\R^d}(1+n)})=\ln (\frac{1}{c_{\infty}(n)}),$$
so that 
$$L(\phi_-)\leq n.$$
The next elementary lemma will be useful.
\begin{lemma}\label{triv}
Let $\Omega$ be an arbitrary open subset of $\R^n$ and $\phi \in L^2(\Omega)\cap L^{\infty}(\Omega),$ $\phi \not=0.$ Then
$$|e^{\phi}-1|_2\leq \frac{e^{|\phi|_{\infty}}-1}{|\phi|_{\infty}}|\phi|_2.$$
If moreover $\phi \in H^1(\Omega),$ then $e^{\phi}-1\in H^1(\Omega).$
\end{lemma}
\begin{proof}
To prove the first point we expand
$$|e^{\phi}-1|= |\phi \sum_{n\geq 1}\frac{\phi^{n-1}}{n!}|\leq \frac{|\phi |}{|\phi|_{\infty}}\sum_{n\geq 1}\frac{|\phi|_{\infty}^n}{n!}$$
The last assertion results from the estimate
$$|\nabla e^{\phi}|_2\leq e^{|\phi|_{\infty}}|\nabla \phi|_2.$$
\end{proof}

 Let $F(\phi)=e^{\phi}-1.$ Then $F'(\phi)$ is bounded by a positive constant $K$ on the interval $\lbrack -K_-,K_+\rbrack .$ Let $\lambda > K$ be so that the function $\lambda I-F$ is strictly increasing in  $\lbrack -K_-,K_+\rbrack .$\\
  Let $B_p=\lbrace x\in \R^d, |x|<p\rbrace.$ We consider the auxiliary problem, for a given $\phi \in H^2(B_p)$ satisfying
$-K_-\leq \phi \leq K_+$,
\begin {equation}\label{auxiliary}
 \left\lbrace
    \begin{array}{l}
    \dsp -\Delta \psi+\lambda \psi =\lambda\phi-F(\phi)+n\\
    \dsp \psi_{\vert_{\partial B_p}}=0,
    \end{array}\right.
\end {equation}
and we write $S(\phi)=\psi$.
Since  $F(\phi)\in L^2(B_p),$ $\psi=S(\phi)\in H^1_0\cap H^2(B_p).$ Moreover, one checks easily by the maximum principle that, since $-K_-=\phi_-\leq \phi \leq K_+=\phi_+$, then also
$$-K_-=\phi_-\leq S(\phi) \leq K_+=\phi_+.$$
We now define inductively $\phi_0=\phi_-$, $\phi_{k+1}=S(\phi_k).$ By the maximum principle, one checks that $(\phi_k)$ is increasing, $\phi_{k+1}\geq \phi_k, k\in \N$, and that
$$\phi_-\leq \phi_k\leq \phi_+, k\in\N.$$
Moreover, $S: L^2(B_p)\rightarrow L^2(B_p)$ is continuous since $z\mapsto \lambda z-F(z)$ is Lipschitz. The sequence $(\phi_k)$ is increasing and bounded from above. 
It converges almost everywhere to some $\phi$ which belongs to $L^2(B_p)$ since $B_p$ is bounded. By Lebesgue's theorem, the convergence holds also in $L^2(B_p).$ On the other hand, using that $F$ is Lipschitz on  $\lbrack -K_-,K_+\rbrack ,$ one checks that  $(\phi_k)$ is Cauchy in $H^1_0(B_p),$ proving that $\phi \in H^1_0(B_p),$ and is solution of (\ref{elliptic}) in $B_p.$ Moreover  $e^{\phi}-1\in L^2(B_p)$ and $\phi \in H^2(B_p).$ 

 Assuming that $\phi_1$ and $\phi_2$ are two $H^1_0\cap H^2(B_p)$ solutions we deduce immediately that, setting $\phi=\phi_1-\phi_2,$
 $$|\nabla \phi|_2^2+\int_{B_p} (e^{\phi_1}-e^{\phi_2})\phi=0,$$
and we conclude that $\phi =0$ by the monotonicity of the exponential.

To summarize, for any $p\in \N$, we have proven the existence of a unique solution in $H^1_0\cap H^2(B_p)$, (which we will denote $\phi_p$ from now on), of  the elliptic problem
\begin {equation}\label{elliptic2}
 \left\lbrace
    \begin{array}{l}
    \dsp -\Delta \phi+e^{\phi}-1 =n\\
    \dsp \phi_{\vert_{\partial {B_p}}} =0.
    \end{array}\right.
\end {equation}
Moreover, $\phi_p$ satisfies the bounds  :
\begin{equation}\label{uniform}
\phi_-=-K_-\leq \phi_p \leq \phi_+=K_+.
\end{equation}
We derive now a series of (uniform in $p$) estimates on $\phi_p.$
We first notice that for $\phi \geq -K_-,$ one has with $\alpha_0=\frac{1-e^{-K_-}}{K_-},$
\begin{equation}\label{easy}
(e^{\phi}-1)\phi\geq \alpha_0 \phi^2.
\end{equation}
Multiplying (\ref{elliptic2}) by $\phi_p$ and integrating over $B_p$ we thus deduce
\begin{equation}\label{est1}
\int_{B_p}\lbrack \frac{\alpha_0}{2} |\phi_p|^2+|\nabla \phi_p|^2 \rbrack \leq \frac{1}{2\alpha_0}|n|^2_2.
\end{equation}
Multiplying (\ref{elliptic2}) by $-\Delta \phi_p$ and integrating over $B_p$ we obtain
\begin{equation}\label{est2}
\int_{Bp}|\Delta \phi_p|^2+\int_{Bp}e^{\phi_p}|\nabla \phi_p|^2 =\int_{Bp}\nabla \phi_p\cdot\nabla n.
\end{equation}
Finally we integrate (\ref{elliptic2}) against $(e^{\phi_p}-1)$ to get
\begin{equation}\label{est3}
%\left\brace
\int_{B_p}e^{\phi_p}|\nabla \phi_p|^2 +\frac{1}{2}\int_{B_p}(e^{\phi_p}-1)^2\leq \frac{1}{2} \int_{B_p}n^2.
\end{equation}
Adding (\ref{est1}), (\ref{est2}), (\ref{est3}) we obtain
 \begin{eqnarray}
\nonumber
\int_{B_p}\lbrack  \frac{\alpha_0}{2}|\phi_p|^2  +\frac{1}{2}|\nabla \phi_p|^2 +|\Delta \phi_p|^2 +2e^{\phi_p}|\nabla \phi_p|^2 +\frac{1}{2}(e^{\phi_p}-1)^2 \rbrack\\
\label{est4}
\leq \frac{1}{2\alpha_0}|n|^2_2+\frac{1}{2}|\nabla n|^2_2+\frac{1}{2}|n|_2^2 .
\end{eqnarray}
Now we extend $\phi_p$ outside $B_p$ by $0$ to get a $H^1(\R^d)$ function  $\tilde{\phi_p}.$ Obviously $\tilde{\phi_p}$ satisfies the bound (\ref{est1}) and (\ref{est3}) and up to a subsequence, $\tilde{\phi_p}$
converges weakly in $H^1(\R^d)$, strongly in $L^2_{loc}(\R^d)$, and almost everywhere to some function $\phi \in H^1(\R^d)$ which satisfies the bound (\ref{uniform}).\\
Let us prove that $\phi$ is solution of the elliptic equation $L\phi =n$ (see (\ref{elliptic})). Let $\chi \in \mathfrak{D}(\R^d).$ Then supp $\chi \subset B_p$ for some $p$. Since $\phi_p$ solves $L\phi_p =n$ in
$B_p$, one has
$$\int_{B_p}\nabla \phi_p\cdot\nabla \chi +\int_{B_p}(e^{\phi_p}-1)\chi=\int_{B_p}n\chi,$$
and thus
$$\int_{\R^d}\nabla \tilde{ \phi_p}\cdot\nabla \chi +\int_{\R^d}(e^{\tilde{\phi_p}}-1)\chi=\int_{\R^d}n\chi.$$
One then infers by  weak convergence and Lebesgue theorem (using Lemma \ref{triv}) that
$$\int_{\R^d}\nabla \phi\cdot\nabla \chi +\int_{\R^d}(e^{\phi}-1)\chi=\int_{\R^d}n\chi,$$
proving that $\phi$ solves
\begin{equation}\label{phiell}
-\Delta \phi+e^{\phi}-1=n.
\end{equation}
Uniqueness is derived by the same argument used for $\phi_p.$ By passing to the limit in (\ref{est1}) and (\ref{est3}) one obtains the corresponding estimates for $\phi.$ Since $e^{\phi}-1-n \in L^2(\R^d)$, $\phi \in H^2(\R^d)$ and (\ref{est2}) for $\phi$ follows. Finally the estimate (ii) in Proposition \ref{prop1}, that is (\ref{est4}) for $\phi$, follows by adding the previous estimates.\\
We now prove the higher regularity estimates (iii) assuming  that
$n\in H^s(\R^d)$. From the continuity of the Riesz transforms (see
Remark \ref{triv}), it is enough to control $\abs{\Lambda^s
  \Delta\phi}_2$. We get from \eqref{phiell} that
\begin{eqnarray*}
\abs{\Lambda^s \Delta\phi}_2&\leq&
\abs{n}_{H^s}+\abs{e^{\phi}-1}_{H^s}\\
&\leq& \abs{n}_{H^s}+C(\abs{\phi}_\infty) \abs{\phi}_{H^s},
\end{eqnarray*}
the second line being a consequence of Moser's inequality. We can
therefore deduce the result by a simple induction using the first two
points of the proposition.
\end{proof}
\begin{remark}
\label{am}
One easily checks that the energy (see \eqref{hamiltonian}) makes sense for $(n,{\bf v},\phi) \in H^1(\R^d)\times H^1(\R^d)^d\times L^2(\R^d)\cap L^{\infty}(\R^d).$ In fact one has then (recalling that $\rho =1+n$),
$$\rho \phi-(e^{\phi}-1)= \rho \phi-(\phi+g)=n\phi-g, \qquad g \in L^1(\R^d),$$
and on the other hand
$$\rho |{\bf v}|^2= |{\bf v}|^2+n|{\bf v}|^2 \in L^1(\R^d),$$
by Sobolev embedding.
\end{remark}

\subsection{Local well-posedness}

We establish here the local well-posedness of the Cauchy problem for \eqref{ivp-eulerpoisson}.

\begin{theorem}\label{ivp}
Let $s>\frac{d}{2}+1$, $n_0\in H^{s-1}(\R^d)$, ${\bf v}_0 \in
H^s(\R^d)^d$ such that $\dsp \inf_{\R^d}1+n_0 >0$. \\
There exist $T>0$ and a unique solution $(n,{\bf v})\in C(\lbrack 0,T\rbrack; H^{s-1}(\R^d)\times H^s(\R^d)^d)$ of \eqref{ivp-eulerpoisson} such that $|(1+n)^{-1}|_{\infty,T}<1$ and $\phi \in
C(\lbrack 0,T\rbrack; H^{s+1}(\R^d))$. \\
Moreover the energy \eqref{hamiltonian} is conserved on $\lbrack 0,T\rbrack$ and so is the impulse (if $a=0$).
\end{theorem}

\begin{proof}
Solving \eqref{elliptic} for $\phi$, we set $\nabla \phi =F(n)=\nabla L^{-1}(n)$ and rewrite \eqref{ivp-eulerpoisson} as
\begin{equation}
    \label{ivp-eulerpoissonbis}
    \left\lbrace
   \begin{array}{lcl}
   	\dsp\partial_tn+\nabla\cdot{\bf v}+\nabla\cdot(n{\bf v})=0,\\
	\dsp\partial_t{\bf v}+({\bf v}\cdot\nabla){\bf v}+F(n)+a{\bf e}\wedge {\bf v}=0,\\
		n(\cdot,0)=n_0,\\
	{\bf v}(\cdot,0)={\bf v}_0.
    \end{array}\right.
\end{equation}
We derive first energy estimates. We apply the operator $\Lambda ^{s-1}$ to the first equation in \eqref{ivp-eulerpoissonbis},  $\Lambda ^s$ to the second one and take the $L^2$ scalar product with $\Lambda^{s-1}n$ and $\Lambda ^s{\bf v}$ respectively. Using the Kato-Ponce commutator estimates \cite{KP} and integration by parts in the first equation, we obtain (note that the skew-adjoint term does not play a role here):
\begin{equation}
\label {energyest}
 \left\lbrace
   \begin{array}{lcl}
   	\dsp \frac{1}{2}\frac{d}{dt}|\Lambda^{s-1} n|_2^2\leq |\Lambda^s {\bf v}|_2(1+|\Lambda^{s-1} n|_2)+C|\nabla{\bf v}|_{\infty}|\Lambda^{s-1} n|_2^2,\vspace{2mm}\\
	\dsp\frac{1}{2}\frac{d}{dt}|\Lambda^s {\bf v}|_2^2\leq C|\nabla{\bf v}|_{\infty}|\Lambda^s{\bf v}|_2^2 +C|\Lambda^s F(n)|_2|\Lambda^s{\bf v}|_2.
		 \end{array}\right.
\end{equation}
As soon as $|(1+n)^{-1}|_{L^{\infty}_{(x,t)}} <\infty$, using Proposition \ref{prop1} (for the existence of $\phi$ given $n$) we infer that
\begin{equation}\label{estF}
|\Lambda^s F(n)|_2\leq F_{s-1}(I_1(n),\abs{n}_\infty)\big(1+\abs{n}_{H^{s-1}}\big).
\end{equation}
One has plainly that
$$I_1(n)\leq C|n|^2_{H^1}(\frac{1}{1-|n|_{\infty}}+1)\leq C |n|^2_{H^1}(\frac{2}{c_0}+1),$$
if $1+n \geq \frac{c_0}{2}$.\\
Gathering those inequalities, we obtain the system of differential
inequalities for  $|n(\cdot, t)|_{H^{s-1}}$ and $|{\bf v}(\cdot,
t)|_{H^s}$, and provided that $1-\abs{n}_\infty \geq \frac{c_0}{2}$,
\begin{equation}
\label {diffineq}
 \left\lbrace
   \begin{array}{lcl}
   	\dsp \frac{d}{dt}| n|_{H^{s-1}}\leq C_1(|{\bf
          n}|_{H^{s-1}},|{\bf v}|_{H^s})\\
	\dsp\frac{d}{dt}|{\bf v}|_{H^s}\leq  C_2(|{\bf
          n}|_{H^{s-1}},|{\bf v}|_{H^s})
		 \end{array}\right.
\end{equation}
where $C_1$ and $C_2$ are smooth functions.
Let $\alpha(t)=|n(\cdot, t)|_{H^{s-1}}$ and $\beta(t)=|{\bf v}(\cdot, t)|_{H^{s}}$ and consider the differential system:
\begin{equation}
\label {diffsys}
 \left\lbrace
   \begin{array}{lcl}
   	\dsp \alpha'=C_1(\alpha,\beta)),\\
	\dsp \beta '=C_2(\alpha,\beta)   .
		 \end{array}\right.
\end{equation}
Let $(A,B)$ be the local solution of \eqref {diffsys} with initial data $(A_0,B_0)=(|n_0|_{H^{s-1}}, |{\bf v}_0|_{H^s})$ satisfying $1+n_0\geq c_0$. The solution to \eqref{diffsys} exists on a time interval which length depends only on $(A_0, B_0)$.\\
Coming back to \eqref {diffineq}, one deduces that $(|n(\cdot, t)|_{H^{s-1}},|{\bf v}(\cdot, t)|_{H^s})$ 
is bounded from above by $(A,B)$ 
 on a time interval $I$ which length depends only on  $(|n_0|_{H^{s-1}},|{\bf v}_0|_{H^s})$ (and possibly shortened
to ensure that $1+n\geq \frac{c_0}{2}$ by continuity).\\
To complete the proof we need to smooth out \eqref{eulerpoisson}. This can be done for instance by truncating the high frequencies, that is using $\chi(jD)$ where $\chi$ is a cut-off function and $j=1,2,\dots$. We obtain an ODE system in $H^{s-1}\times H^s.$  The energy estimates are derived as above and one passes to the limit in a standard way.

Finally, the conservation of both the energy and the impulse (if $a=0$) is obvious since the functional setting of Theorem \ref{ivp} allows to justify their formal derivation.
\end{proof}

\section{The long wave limit of the Euler-Poisson system}\label{LWEP}

In order to justify the Zakharov-Kuznetsov equation as a long wave limit of the Euler-Poisson system with an applied magnetic field, we have to introduce an appropriate scaling.

We  set ${\bf v}= (v_x,v_y,v_z).$  Laedke and Spatschek  \cite{LaS} derived formally the Zakharov-Kuznetsov equation by looking  for approximate solutions of \eqref{eulerpoisson} of the form
\begin{equation}\label{rescaled}
\begin{split}
n^{\e}&=\e\,n^{(1)}(\e^{1/2}(x-t), \e^{1/2}y,\e^{1/2} z,\e^{3/2}t) +\e^2n^{(2)}+\e^3n^{(3)}\\
\phi^{\e}&=\e\,\phi^{(1)}(\e^{1/2}(x-t), \e^{1/2}y, \e^{1/2} z, \e^{3/2}t) +\e^2\phi^{(2)}+\e^3\phi^{(3)}\\
v_x^{\e}&=\e\,v_x^{(1)}(\e^{1/2}(x-t), \e^{1/2}y, \e^{1/2} z, \e^{3/2}t) +\e^2v_x^{(2)}+\e^3v_x^{(3)}\\
v_y^{\e}&=\e^{3/2}\,v_y^{(1)}(\e^{1/2}(x-t), \e^{1/2}y, \e^{1/2} z,\e^{3/2}t) +\e^2v_y^{(2)}+\e^{5/2}v_y^{(3)}\\
v_z^{\e}&=\e^{3/2}\,v_z^{(1)}(\e^{1/2}(x-t), \e^{1/2}y, \e^{1/2} z, \e^{3/2}t) +\e^2v_z^{(2)}+\e^{5/2}v_z^{(3)}.
\end{split}
\end{equation}
The asymptotic analysis of the Euler-Poisson system is easier to handle if we work with rescaled variables and unknowns adapted to this ansatz. More precisely, if we introduce
\begin{equation*}
\tilde{x}=\epsilon^{1/2} x,\quad \tilde{y}=\epsilon^{1/2}y,\quad  \tilde{z}=\epsilon^{1/2}z,\quad\tilde{t}=\epsilon ^{1/2} t,\quad
\tilde{n}=\epsilon n,\quad\tilde{\phi}=\epsilon \phi, \quad \tilde\bv=\epsilon \bv,
\end{equation*}
the Euler-Poisson equation  \eqref{eulerpoisson}  becomes (dropping the tilde superscripts),
\begin{equation}
    \label{scaEP}
    \left\lbrace
   \begin{array}{lcl}
   	\dsp\partial_tn+\nabla\cdot ((1+\eps n){\bf v}) =0,\\
	\dsp\partial_t{\bf v}+\eps ({\bf v}\cdot \nabla){\bf v}+\nabla\phi+a\eps^{-1/2}{\bf e}\wedge {\bf v}=0,\\
	\dsp -\epsilon ^2\Delta\phi+e^{\epsilon\phi}-1=\epsilon n,
    \end{array}\right.
\end{equation}
and the ZK equation is derived by looking for approximate solutions to this system under the form \footnote{Actually we will not use the third order profiles.}
\begin{equation}
\label{ZKansatz}
\begin{split}
n^{\e}&=\,n^{(1)}(x-t, y,z,\e t) +\e n^{(2)}+\e^2n^{(3)}\\
\phi^{\e}&=\,\phi^{(1)}(x-t, y,  z, \e t) +\e\phi^{(2)}+\e^2\phi^{(3)}\\
v_x^{\e}&=\,v_x^{(1)}(x-t, y, z, \e t) +\e v_x^{(2)}+\e^2v_x^{(3)}\\
v_y^{\e}&=\e^{1/2}\,v_y^{(1)}(x-t, y,  z,\e t) +\e v_y^{(2)}+\e^{3/2}v_y^{(3)}\\
v_z^{\e}&=\e^{1/2}\,v_z^{(1)}(x-t, y, z, \e t) +\e v_z^{(2)}+\e^{3/2}v_z^{(3)}.
\end{split}
\end{equation}
\subsection{The Cauchy problem revisited}
It is easily checked that when applied to \eqref{scaEP}, Theorem \ref{ivp} provides an existence time
which is of order $O(1)$
with respect to $\eps$, while the time scale $O(1/\eps)$ must be reached to observe the dynamics of the 
Zakharov-Kuznetsov
equation that occur along the slow time scale $\eps t$. \\
As explained in the Introduction, we therefore need, in order to justify the Zakharov-Kuznetsov equation as a long wave limit of the Euler-Poisson system, to solve the Cauchy problem associated to \eqref{scaEP} on a time interval of order $O(1/\eps)$. In order to do so, we will consider \eqref{scaEP} as a perturbation of a hyperbolic quasilinear system and give a proof which does not use the smoothing effect of the $\phi$ equation for a fixed $\epsilon$. This is the reason why more regularity is required on the initial data in the statement of
the theorem below. Before stating it, let us introduce the space $H^{s+1}_\eps$ which is the standard Sobolev space $H^{s+1}(\R^d)$ endowed with the norm 
\begin{equation}\label{defHseps}
\forall s\geq 0,\quad \forall f\in H^{s+1},\quad \abs{f}_{H^{s+1}_\eps}^2=\abs{f}_{H^s}^2+\eps\abs{\nabla f}_{H^s}^2;
\end{equation}
the presence of the small parameter $\eps$ in front of the second term is here to make this norm adapted
to measure the smoothing effects of the $\phi$ equation; the fact that these smoothing effects are small
explains why Theorem \ref{ivp}, which relies on them, provides an existence time much too small to observe
the dynamics of the Zakharov-Kuznetsov equation.
\begin{theorem}\label{thmIVPunif}
Let $s>\frac{d}{2}+1$ and $n_0\in H^{s}(\R^d)$, ${\bf v}_0 \in H^{s+1}(\R^d)^d$ such that $1+n_0\geq c_0$ for some $c_0>0$. \\
Then there exists $\underline{T}>0$ such that for all $\eps\in (0,1)$, there is 
 a unique solution $(n^\eps,{\bf v}^\eps,\phi^\e)\in C(\lbrack
 0,\frac{\underline{T}}{\eps}\rbrack; H^{s}(\R^d)\times
 H^{s+1}_\eps(\R^d)^d\times H^{s+1}(\R^d))$ of \eqref{scaEP} such that $1+\eps n^{\epsilon}>c_0/2$. \\
Moreover the family $(n^\eps,\bv^\eps,\nabla\phi^\eps)_{\eps\in (0,1)}$ is uniformly bounded in $H^s\times H^{s+1}_\eps\times H^{s-1}$.
\end{theorem}
\begin{proof}
{\bf Step 1. Preliminary results.} If we differentiate the third equation of (\ref{scaEP}) with respect to $\partial_j$ ($j=x,y,z,t$), we get
\begin{equation}\label{linearize}
\Me\partial_j\phi=\partial_j n \quad\mbox{ with }\quad \Me=-\eps\Delta+e^{\eps\phi};
\end{equation}
the operator $\Me$ plays a central role in the energy estimates below, and we state here some basic 
estimates. The first is  that $(u,\Me u)^{1/2}$ defines a norm which is uniformly equivalent to $\abs{\cdot}_{H^1_\eps}$; more precisely, for all $u\in H^1(\R^d)$,
\begin{equation}\label{equiv}
(u,\Me u)\leq C(\abs{\phi}_\infty)\abs{u}_{H^1_\eps}^2
\quad\mbox{ and }\quad
\abs{u}_{H^1_\eps}^2\leq C(\abs{\phi}_{\infty})(u,\Me u).
\end{equation}
We also have for all $u\in H^1(\R^d)$ and $f\in W^{1,\infty}(\R^d)$,
$$
\big(u,f\Me u\big)\leq C(\abs{\phi}_\infty,\abs{f}_\infty,\sqrt{\eps}\abs{\nabla f}_\infty)\,\abs{u}_{H^1_\eps}^2,
$$
and, for all $u\in H^1(\R^d)$ and $f\in W^{2,\infty}(\R^d)$,
$$
\big( u,[f\partial_j,\Me] u\big)\leq C(\abs{\phi}_{W^{1,\infty}},\abs{f}_{W^{1,\infty}},\sqrt{\eps}\abs{\nabla \partial_jf}_\infty)\,\abs{u}_{H^1_\eps}^2\qquad (j=x,y,z);
$$
these two estimates are readily obtained from the definition of $\Me$ and integration by parts.\\
Let us finally prove that $\Me$ is invertible and give estimates on its inverse. 
\begin{lemma}\label{lemmest}
Let $\phi\in L^\infty(\R^d)$. Then $\Me:H^2(\R^d)\to L^2(\R^d)$ is an isomorphism and
$$
\forall v\in L^2(\R^d), \qquad e^{-\frac{\eps}{2}\abs{\phi}_\infty}\abs{\Me^{-1}v}_2+\sqrt{\eps} \abs{\nabla \Me^{-1}v}_2\leq e^{\frac{\eps}{2}\abs{\phi}_\infty}\abs{v}_2.
$$
If moreover $t_0>d/2$, $s\geq 0$ and $\phi\in H^{t_0+1}\cap H^{s}(\R^d)$, then we also have
$$
\forall v\in H^s(\R^d), \qquad \abs{\Me^{-1}v}_{H^s}+\sqrt{\eps} \abs{\nabla \Me^{-1}v}_{H^s}\leq C(\abs{\phi}_{H^{t_0+1}\cap H^s})\abs{v}_{H^s}.
$$
\end{lemma}
\begin{proof}[Proof of the lemma]
The invertibility property of $\Me$ follows classically from Lax-Milgram's theorem, and the first estimate of the lemma
follows from the coercivity property of $\Me$, namely,
\begin{equation}\label{coerc}
e^{-\eps\abs{\phi}_\infty}\abs{u}_2^2+\eps \abs{\nabla u}_2^2\leq (\Me u,u).
\end{equation}
In order to prove the higher order estimates, let us write $u=\Me^{-1}v$. We have by definition $(-\eps\Delta+e^{\eps\phi})u=v$, so that applying $\Lambda^s$ on both sides, we get
\begin{equation}\label{etcomm}
(-\eps\Delta+e^{\eps\phi})\Lambda^s u=\Lambda^s v-[\Lambda^s,e^{\eps\phi}]u.
\end{equation}
Using the first estimate, and recalling that $u=\Me^{-1}v$, we deduce that
$$
\abs{\Me^{-1}v}_{H^s}+\sqrt{\eps} \abs{\nabla \Me^{-1}v }_{H^s}\leq e^{\eps\abs{\phi}_\infty}\big(\abs{v}_{H^s}+\abs{[\Lambda^s,e^{\eps\phi}]\Me^{-1}v}_2\big).
$$
Now, the use Kato-Ponce and Coifman-Meyer commutator estimates, and Moser's inequality, shows that
\begin{eqnarray*}
\lefteqn{\abs{[\Lambda^s,e^{\eps\phi}]\Me^{-1}v}_2\leq\eps C(\abs{\phi}_\infty)}\\
&\times& \Big( \abs{\phi}_{H^{t_0+1}}\abs{\Me^{-1}v}_{H^{s-1}}
+\langle \abs{\phi}_{H^s}\abs{\Me^{-1}v}_{H^{t_0}}\rangle_{s>t_0+1}\Big),
\end{eqnarray*}
and we get therefore
\begin{eqnarray*}
\abs{\Me^{-1}v}_{H^s}+\sqrt{\eps} \abs{\nabla \Me^{-1}u }_{H^s}&\leq& C(\abs{\phi}_\infty)
\big(\abs{v}_{H^s}+\abs{\phi}_{H^{t_0+1}}\abs{\Me^{-1}v}_{H^{s-1}}\\
& &\indent\indent+\langle \abs{\phi}_{H^s}\abs{\Me^{-1}v}_{H^{t_0}}\rangle_{s>t_0+1}\big),
\end{eqnarray*}
and the results follows by a continuous induction.
\end{proof}

\noindent
{\bf Step 2. $L^2$ estimates for a linearized system.} Let $T>0$ and $(\underline{n},\underline{\bv})\in L^\infty([0,T];W^{1,\infty}(\R^d))\cap W^{1,\infty}([0,T];L^\infty(\R^d))$ be such that $1+\eps\underline{n}\geq c_0>0$ uniformly on $[0,T]$, and $ \underline{\phi}\in W^{1,\infty}$ with $\partial_t\underline{\phi}\in L^{\infty}.$ Let us consider a classical solution $(n,\bv,\phi)$ of the linear system
\begin{equation}
    \label{EPlin}
    \left\lbrace
   \begin{array}{lcl}
   	\dsp\partial_tn+(1+\eps\underline{n})\nabla\cdot{\bf v}+\eps\underline{\bv}\cdot \nabla n =\eps f,\\
	\dsp\partial_t{\bf v}+\eps (\underline{\bf v}\cdot \nabla){\bf v}+\nabla\phi+a\eps^{-1/2}{\bf e}\wedge {\bf v}=\eps {\bf g},\\
	\dsp \Meb \nabla\phi=\nabla n+\eps {\bf h}.
    \end{array}\right.
\end{equation}
We want to prove here that 
\begin{eqnarray}
\nonumber
\sup_{[0,T]}\big(\abs{n}_2^2+\abs{\bv}_{H^1_\eps}^2\big)& \leq& \exp(\eps C_0 T)\\
\label{EstL2}
&\times& \big(\abs{n_{\vert_{t=0}}}_2^2+ \abs{\bv_{\vert_{t=0}}}_{H^1_\eps}^2\big)+\eps T(\abs{f}^2_{L^2_T}+\abs{{\bf g}}^2_{H^1_{\eps,T}}+\abs{{\bf h}}_{L^2_T}^2)\big),
\end{eqnarray}
with $C_0=C(\frac{1}{c_0},\abs{(\underline{n},\underline{\bv},\phi)}_{W^{1,\infty}_T},\abs{\dt ( \underline{n}, \underline{\bv}, \underline{\phi})}_{L^\infty_T},\sqrt{\eps}\abs{\nabla\nabla\cdot \ubv}_{L^\infty_T})$.\\
Taking the $L^2$-scalar product with $\frac{1}{1+\eps\underline{n}}n$, we get
$$
\big(\frac{1}{1+\eps\un}\dt n,n\big)+\big(\nabla\cdot \bv,n\big)+\eps\big(\frac{\underline{\bv}}{1+\eps\un}\cdot\nabla n,n)=\eps \big(\frac{f}{1+\eps\un},n\big),
$$
which can be rewritten under the form
\begin{eqnarray}
\nonumber
\frac{1}{2}\dt \big(\frac{1}{1+\eps\un}n,n\big)+\eps\frac{1}{2}\big(\frac{\dt \un}{(1+\eps\un)^2}n,n\big)
+(\nabla\cdot \bv,n)\\
 \label{estL21}
-\eps\frac{1}{2}\big(n,\nabla\cdot(\frac{\underline{\bv}}{1+\eps\un})n\big)
=\eps \big(\frac{f}{1+\eps\un},n\big).
\end{eqnarray}
We now  take the scalar product of the second equation with $\Meb \bv$ to obtain, after recalling that $\Meb \nabla\phi= \nabla n+\eps{\bf h}$, 
\begin{eqnarray*}
\big(\Meb\dt \bv,\bv\big)+\eps\big(\Meb\underline{\bv}\cdot\nabla \bv,\bv\big)
+\big(\nabla n,\bv\big)\\
=-\eps\big({\bf h},\bv\big)+\eps \big(\Meb{\bf g},\bv\big),
\end{eqnarray*}
or equivalently
\begin{eqnarray}
\nonumber
\lefteqn{\frac{1}{2}\dt \big(\Meb\bv,\bv\big)-\frac{1}{2}\eps\big(\dt \underline{\phi}e^{\eps\underline{\phi}}\bv,\bv\big)
-\frac{1}{2}\eps\big(\bv,(\nabla\cdot\ubv)\Meb \bv\big)}\\
& &-\frac{1}{2}\eps \big(\bv,[\ubv\cdot\nabla,\Meb]\bv\big)+(\nabla n,\bv)
\label{estL22}
=-\eps\big({\bf h},\bv\big)+\eps \big(\Meb{\bf g},\bv\big).
\end{eqnarray}
Adding (\ref{estL22}) to (\ref{estL21}), we get therefore
\begin{eqnarray*}
\lefteqn{\frac{1}{2}\dt \big(\frac{1}{1+\eps\un}n,n\big)+\frac{1}{2}\dt \big(\Meb\bv,\bv\big)=
\eps\frac{1}{2}\big([\frac{\dt \un}{(1+\eps\un)^2}+\nabla\cdot(\frac{\underline{\bv}}{1+\eps\un})]n,n\big)}\\
& &+\frac{1}{2}\eps\big(\dt \underline{\phi}e^{\eps\underline{\phi}}\bv,\bv\big)
+\frac{1}{2}\eps\big(\bv,(\nabla\cdot\ubv)\Meb \bv\big)+\frac{1}{2}\eps\big(\bv,[\ubv\cdot\nabla,\Meb]\bv\big)\\
& &+\eps \big(\frac{1}{1+\eps\un}f,n\big)-\eps\big({\bf h},\bv\big)+\eps \big(\Meb{\bf g},\bv\big).
\end{eqnarray*}
All the terms on the right-hand-side are easily controlled (with the help of Step 1 for third and fourth terms) to obtain 
\begin{eqnarray*}
\dt \big\{ \!\big(\frac{n}{1+\eps\un},n\big)\!\!+\!\! \big(\Meb\bv,\bv\big)\!\big\}\!\leq \eps C(\frac{1}{c_0}\!,\abs{(\underline{n},\underline{\bv},\underline{\phi})}_{W^{1,\infty}},\abs{\dt ( \underline{n}, \underline{\bv},\underline{\phi})}_\infty,\eps \abs{\nabla \nabla\!\cdot\!\ubv}_{2})\\
\times\big( (\frac{1}{1+\eps\un}n,n)+\abs{\bv}_{H^1_\eps}^2+\abs{f}_2^2+\abs{{\bf g}}_{H^1_\eps}^2+\abs{{\bf h}}_2^2\big).
\end{eqnarray*}
Using (\ref{equiv}) and a Gronwall inequality, we readily deduce (\ref{EstL2}). 

\noindent
{\bf Step 3. $H^s$ estimates for a linearized system.} We want to prove here that 
the solution $(n,\bv,\phi)$ to (\ref{EPlin}) satisfies, for all $s\geq t_0+1$,
\begin{eqnarray}
\nonumber
\sup_{[0,T]}\big(\abs{n}_{H^s}^2+\abs{ \bv}_{H^{s+1}_\eps}^2\big)& \leq& 
\exp(\eps C_s T) \times \Big(\abs{n_{\vert_{t=0}}}_{H^s}^2+ \abs{ \bv_{\vert_{t=0}}}_{H^{s+1}_\eps}^2\big)\\
\label{ests}
& &\indent+\eps T(\abs{f}^2_{H^s_T}+\abs{{\bf g}}^2_{{H}^{s+1}_{\eps,T}}+\abs{ {\bf h}}_{H^s_T}^2+\abs{n}_{H^s}^2+\abs{ \bv}_{H^{s+1}_\eps}^2)\Big),
\end{eqnarray}
with $C_s= C(\frac{1}{c_0},\abs{\un}_{H^s_T},\abs{\ubv}_{H^{s+1}_{\eps,T}},\abs{\underline{\phi}}_{H^s_T},\abs{\dt ( \underline{n}, \underline{\bv}, \underline{\phi})}_{L^\infty_T})$.\\
Applying $\Lambda^s$ to the three equations of (\ref{EPlin}), and writing $\tilde n=\Lambda^s n$, $\tilde \bv=\Lambda^s \bv$, and $\tilde\phi=\Lambda^s \phi$, we get
\begin{equation}
    \label{EPlinHs}
    \left\lbrace
   \begin{array}{lcl}
   	\dsp\partial_t\tilde n+(1+\eps \underline{n})\nabla\cdot \tilde{\bf v} +\eps\ubv\cdot\nabla\tilde n=\eps\tilde f,\\
	\dsp\partial_t\tilde{\bf v}+\eps (\underline{\bf v}\cdot \nabla)\tilde{\bf v}+\nabla\tilde\phi+a\eps^{-1/2}{\bf e}\wedge \tilde {\bf v}=\eps \tilde {\bf g},\\
	\dsp \Meb\nabla\tilde\phi=\nabla\tilde n+\eps\tilde {\bf h}.
    \end{array}\right.
\end{equation}
with 
\begin{eqnarray*}
\tilde f&=&\Lambda^s f- [\Lambda^s,\underline{n}]\nabla\cdot \bv-[\Lambda^s,\ubv]\cdot\nabla n,\\
\tilde {\bf g}&=&\Lambda^s {\bf g}-[\Lambda^s,\underline{\bv}]\cdot\nabla \bv,\\
\tilde{\bf h}&=&\Lambda^s{\bf h}-\frac{1}{\eps}[\Lambda^s,\Meb]\nabla\phi.
\end{eqnarray*}
From Step 2, we get that
\begin{eqnarray*}
\sup_{[0,T]}\big(\abs{\tilde n}_2^2+\abs{\tilde \bv}_{H^1_\eps}^2\big) \leq 
\exp (\eps C_0 T) \\
\times \big(\abs{\tilde n_{\vert_{t=0}}}_2^2+ \abs{\tilde \bv_{\vert_{t=0}}}_{H^1_\eps}^2\big)+\eps T(\abs{\tilde f}^2_{L^2_T}+\abs{\tilde {\bf g}}^2_{H^1_{\eps,T}}+\abs{\tilde {\bf h}}_{L^2_T}^2)\big).
\end{eqnarray*}
Now, by standard commutator estimates, we have for all $s\geq t_0+1$,
$$
\abs{\tilde f}_2+\abs{\tilde{\bf g}}_{H^1_\eps}\leq \abs{f}_{H^s}+\abs{{\bf g}}_{H^{s+1}_\eps}+
\big(\abs{\un}_{H^s}+\abs{\ubv}_{H^{s+1}_\eps}\big)\times \big(\abs{\bv}_{H^{s+1}_\eps}+\abs{n}_{H^s}\big)
$$
and
\begin{eqnarray*}
\abs{\tilde{\bf h}}_2&\leq& \abs{{\bf h}}_{H^s}+C(\abs{\underline{\phi}}_{H^s})\abs{\nabla\phi}_{H^{s-1}},\\
&\leq& C(\abs{\underline{\phi}}_{H^s})\big(\abs{{\bf h}}_{H^s}+\abs{n}_{H^s}\big),
\end{eqnarray*}
where we used Lemma \ref{lemmest} to get the last inequality. We can now directly deduce (\ref{ests})
from the Sobolev embedding $W^{1,\infty}(\R^d)\subset H^s(\R^d)$ ($s\geq t_0+1$).

\noindent
{\bf Step 4. End of the proof.} The exact solution provided by Theorem \ref{ivp} solves (\ref{EPlin})
with $(\un,\underline{\bv},\underline{\phi})=(n,\bv,\phi)$ and $f=0$, ${\bf g}=0$, ${\bf h}=0$;
we deduce therefore from Step 3 that it
satisfies the estimate
$$
\sup_{[0,T]}\big(\abs{n}_{H^s}^2+\abs{ \bv}_{H^{s+1}_\eps}^2\big)\leq 
\exp(\eps  C_s T) \times \Big(\abs{n_{\vert_{t=0}}}_{H^s}^2+ \abs{ \bv_{\vert_{t=0}}}_{H^{s+1}_\eps}^2\big)
+\eps T(\abs{n}_{H^s}^2+\abs{ \bv}_{H^{s+1}_\eps}^2)\Big),
$$
with 
\begin{eqnarray*}
C_s&=& C(\frac{1}{c_0},\abs{n}_{H^s_T},\abs{\bv}_{H^{s+1}_{\eps,T}},\abs{{\phi}}_{H^s},\abs{\dt ( {n}, {\bv}, \phi)}_{L^\infty_T})\\
&=& C(\frac{1}{c_0},\abs{n}_{H^s_T},\abs{\bv}_{H^{s+1}_{\eps,T}}),
\end{eqnarray*}
where we used (\ref{linearize}) and the equation to control the time derivatives in terms of space derivatives,
and Proposition \ref{prop1} to control the $L^2$-norm of $\phi$ (control on $\nabla\phi$ being provided by Lemma \ref{lemmest}). This provides us with a $\underline{T}>0$ such that $\abs{n}_{H^s}^2+\abs{ \bv}_{H^{s+1}_\eps}^2$ remains uniformly bounded with respect to $\eps$ on $[0,\underline{T}/\eps]$.
\end{proof}

\subsection{The ZK approximation to the Euler-Poisson system}\label{ZKA}

We construct here an approximate solution to (\ref{scaEP}) based on the ZK equation.
Following Laedke and Spatschek  \cite{LS}, but with the rescaled variables and unknowns introduced
in (\ref{rescaled}),  we look for approximate solutions of \eqref{scaEP} under the form (\ref{ZKansatz}),
namely,
\begin{equation}
\label{solapp}
\begin{split}
n^{\e}&=\,n^{(1)}(x-t, y,z,\e t) +\e n^{(2)}\\
\phi^{\e}&=\,\phi^{(1)}(x-t, y,  z, \e t) +\e\phi^{(2)}\\
v_x^{\e}&=\,v_x^{(1)}(x-t, y, z, \e t) +\e v_x^{(2)}\\
v_y^{\e}&=\e^{1/2}\,v_y^{(1)}(x-t, y,  z,\e t) +\e v_y^{(2)}\\
v_z^{\e}&=\e^{1/2}\,v_z^{(1)}(x-t, y, z, \e t) +\e v_z^{(2)}.
\end{split}
\end{equation}
\begin{notation}
We will denote by $X$ the variable $x-t$ and by $T$ the slow time variable $\eps t$.
\end{notation}
Plugging this ansatz into the first equation of \eqref{scaEP} we obtain:
\begin{equation}\label{2}
\dt n^{\e}+\nabla\cdot\big(1+\eps n^{\e})\bv^{\e}= \sum_{j=0}^6
\e^{j/2} N^j,
\end{equation}
where
\begin{equation*}
\begin{split}
N^0&=-\frac{\partial} {\partial X}n^{(1)}+\frac{\partial}{\partial X}v_x^{(1)}\\
N^1&=\frac{\partial} {\partial y}v_y^{(1)}+\frac{\partial}{\partial z}v_z^{(1)}\\
N^2&=\frac{\partial} {\partial T}n^{(1)}-\frac{\partial}{\partial X}n^{(2)}
+\frac{\partial} {\partial X}(n^{(1)}v_x^{(1)})
+\frac{\partial} {\partial X}v_x^{(2)}+\frac{\partial} {\partial y}v_y^{(2)}
+\frac{\partial} {\partial z}v_z^{(2)}\\
N^3&=
\frac{\partial} {\partial y}(n^{(1)}v_y^{(1)})
+\frac{\partial} {\partial z}(n^{(1)}v_z^{(1)})\\
N^4&=\frac{\partial} {\partial T}n^{(2)}+\frac{\partial} {\partial X}(n^{(1)}v_x^{(2)})
+\frac{\partial} {\partial y}(n^{(1)}v_y^{(2)})\\
&\;\;\;+\frac{\partial} {\partial z}(n^{(1)}v_z^{(2)})+\frac{\partial} {\partial X}(n^{(2)}v_x^{(1)}),\\
N^5&=\frac{\partial}{\partial y}(n^{(2)}v_y^{(1)})+\frac{\partial}{\partial z}(n^{(2)}v_z^{(1)}),\\
N^6&=\frac{\partial}{\partial X}(n^{(2)}v_x^{(2)})+\frac{\partial}{\partial y}(n^{(2)}v_y^{(2)})+\frac{\partial}{\partial z}(n^{(2)}v_z^{(2)}).
\end{split}
\end{equation*}
Similarly, one has
\begin{equation}\label{3}
\frac{\partial} {\partial t}v_x^{\e}+\eps v_x^{\e}\frac{\partial} {\partial x}v_x^{\e}
+\eps  v_y^{\e}\frac{\partial} {\partial y}v_x^{\e} +\eps  v_z^{\e}\frac{\partial} {\partial z}v_x^{\e}
+\frac{\partial} {\partial x}\phi^{\e}=\sum_{j=0}^6 \e^{j/2}R^1_j 
\end{equation}
with $R_1^1=0$ and
\begin{equation*}
\begin{split}
R_0^1&=-\frac{\partial} {\partial X}v_x^{(1)}+\frac{\partial} {\partial X}\phi^{(1)}\\
R_2^1&=\frac{\partial} {\partial T}v_x^{(1)}-\frac{\partial} {\partial X}v_x^{(2)}
+v_x^{(1)}\frac{\partial} {\partial X}v_x^{(1)}+\frac{\partial} {\partial X}\phi^{(2)}\\
R_3^1&=v_y^{(1)}\frac{\partial} {\partial y} v_x^{(1)}+v_z^{(1)}\frac{\partial} {\partial z}v_x^{(1)}\\
R_4^1&=\frac{\partial} {\partial T}v_x^{(2)}
+\frac{\partial} {\partial X}(v_x^{(1)}v_x^{(2)})+v_y^{(2)}\frac{\partial} {\partial y}v_x^{(1)}+
v_z^{(2)}\frac{\partial} {\partial z }v_x^{(1)}\\
R_5^1&=v_y^{(1)}\frac{\partial}{\partial y}v_x^{(2)} +v_z^{(1)}\frac{\partial}{\partial z}v_x^{(2)}  \\
R_6^1&=v_y^{(2)}\frac{\partial}{\partial y}v_x^{(2)}+v_z^{(2)}\frac{\partial}{\partial z}v_x^{(2)}+v_x^{(2)}\frac{\partial}{\partial X}v_x^{(2)};
\end{split}
\end{equation*}
for the second component of the velocity equation, we get
\begin{equation}\label{4}
\frac{\partial} {\partial t}v_y^{\e}+\eps v_x^{\e}\frac{\partial} {\partial x}v_y^{\e}
+\eps  v_y^{\e}\frac{\partial} {\partial y}v_y^{\e} +\eps  v_z^{\e}\frac{\partial} {\partial z}v_y^{\e}
+\frac{\partial} {\partial y}\phi^{\e}-a\eps^{-1/2}\,v_z^{\e}
=\sum_{j=0}^6 \e^{j/2}R^2_j
\end{equation}
with
\begin{equation*}
\begin{split}
R_0^2&=\frac{\partial} {\partial y}\phi^{(1)}-a v_z^{(1)}\\
R_1^2&=-(\frac{\partial} {\partial X}v_y^{(1)}+a v_z^{(2)})\\
R_2^2&=-\frac{\partial} {\partial X}v_y^{(2)}+\frac{\partial} {\partial y}\phi^{(2)}\\
R_3^2&=\frac{\partial} {\partial T}v_y^{(1)}+v_x^{(1)}
\frac{\partial} {\partial X}v_y^{(1)}\\
R_4^2&=\frac{\partial} {\partial T}v_y^{(2)}+v_x^{(1)}\frac{\partial} {\partial X}v_y^{(2)}+
v_y^{(1)}\frac{\partial} {\partial y}v_y^{(1)}+v_z^{(1)}\frac{\partial} {\partial z}v_y^{(1)}\\
R_5^2&=v_x^{(2)}\frac{\partial}{\partial X}v_y^{(1)}+\frac{\partial}{\partial y}(v_y^{(1)} v_y^{(2)})+v_z^{(1)}\frac{\partial}{\partial z}v_y^{(2)}+v_z^{(2)}\frac{\partial}{\partial z}v_y^{(1)}\\
R_6^2&=v_x^{(2)}\frac{\partial}{\partial
  X}v_y^{(2)}+v_y^{(2)}\frac{\partial}{\partial
  y}v_y^{(2)}+v_z^{(2)}\frac{\partial}{\partial z}v_y^{(2)},
\end{split}
\end{equation*}
while for the third component, the equations are
\begin{equation}\label{5}
\frac{\partial} {\partial t}v_z^{\e}+\e v_x^{\e}\frac{\partial} {\partial x}v_z^{\e}
+ \e v_y^{\e}\frac{\partial} {\partial y}v_z^{\e} + \e v_z^{\e}\frac{\partial} {\partial z}v_z^{\e}
+\frac{\partial} {\partial z}\phi^{\e}+a\e^{-1/2}\,v_y^{\e}=
\sum_{j=0}^6 \e^{j/2} R_j^3
\end{equation}
with
\begin{equation*}
\begin{split}
R_0^3&=\frac{\partial} {\partial z}\phi^{(1)}+a v_y^{(1)}\\
R_1^3&=-\frac{\partial} {\partial X}v_z^{(1)}+a v_y^{(2)}\\
R_2^3&=-\frac{\partial} {\partial X}v_z^{(2)}+\frac{\partial} {\partial z}\phi^{(2)}\\
R_3^3&=\frac{\partial} {\partial T}v_z^{(1)}+v_x^{(1)}\frac{\partial} {\partial X}v_z^{(1)}\\
R_4^3&=\frac{\partial} {\partial T}v_z^{(2)}+v_x^{(1)}\frac{\partial} {\partial X}v_z^{(2)}+
v_y^{(1)}\frac{\partial} {\partial y}v_z^{(1)}+v_z^{(1)}\frac{\partial} {\partial z}v_z^{(1)}\\
R_5^3&=v_x^{(2)}\frac{\partial}{\partial X}v_z^{(1)}+\frac{\partial}{\partial z}(v_z^{(1)}v_z{(2)})+v_y^{(1)}\frac{\partial}{\partial y}v_z^{(2)}+v_y^{(2)}\frac{\partial}{\partial y}v_z^{(1)}\\
R_6^3&=v_x^{(2)}\frac{\partial}{\partial X}v_z^{(2)}+v_y^{(2)}\frac{\partial}{\partial y}v_z^{(2)}+v_z^{(2)}\frac{\partial}{\partial z}v_z^{(2)}.
\end{split}
\end{equation*}
Finally, for the equation on the potential, we obtain
\begin{equation}\label{6}
-\e^2\Delta \phi^{\e}+e^{\e\phi}-1-\e n =\e r^2+\e^2 r^4+\e^3 r^6+O(\e^4)
\end{equation}
with
\begin{equation*}
\begin{split}
r^2&=\phi^{(1)}-n^{(1)}\\
r^4&=-\Delta \phi^{(1)}+\phi^{(2)}+\frac12(\phi^{(1)})^2-n^{(2)}\\
r^6&=-\Delta \phi^{(2)}+\phi^{(1)}\phi^{(2)}+\frac{1}{6}(\phi^{(1)})^3.
\end{split}
\end{equation*}
%where $\Delta^\prime$ denotes the laplacian with respect to the variables $(X,y,z).$\\

We first derive (following essentially \cite{LS}) the equations
corresponding to the successive cancellation of the leading order
remainder terms; we then show that they imply that $n^{(1)}$ must
solve the Zakharov-Kuznetsov equation, and then turn to show  in the
spirit of Ben-Youssef and Colin who considered the one-dimensional
case in  \cite{C-BY} (see also \cite{GuoPu}) that it
is indeed possible to construct an approximate solution \eqref{solapp}
satisfying all the cancellation conditions previously derived. The
consequence is the consistency property of \eqref{solapp} stated in
Proposition \ref{propconsist}.

\subsubsection{Cancellation of terms of order zero in  $\e$}

Canceling the terms $N^0$ and $R^j_0$ ($j=1,2,3$) is equivalent to the
following conditions,
\begin{align}
%\phi^{(1)}&=n^{(1)} \\
\label{paq0}
v_x^{(1)}&=\phi^{(1)}=n^{(1)} \qquad \text{(assuming that}
\;v_x^{(1)},  \; \phi^{(1)},\;n^{(1)}\;\text{vanish as}\;\;|X|\to
+\infty)\\
\label{paq1}
v_y^{(1)}&=-\frac{1}{a}\partial_z n^{(1)}\\
\label{paq2}
v_z^{(1)}&=\frac{1}{a}\partial_y n^{(1)}.
\end{align}

\subsubsection{Cancellation of terms of order  $\e^{1/2}$}

Using \eqref{paq1}-\eqref{paq2}, the cancellation of the terms $N^1$ and
$R^j_1$ ($j=1,2,3$) is equivalent to
\begin{equation}\label{paq3}
v_y^{(2)}=\frac{1}{a^2}\partial_{Xy}^2 n^{(1)}\quad \mbox{ and }\quad
v_z^{(2)}=\frac{1}{a^2}\partial_{Xz}^2 n^{(1)}.
\end{equation}

\subsubsection{Cancellation of terms of order  $\e$}\label{secteps1}
Using the conditions derived above, the cancellation of $N^2$ and
$R^j_2$ ($j=1,2,3$),
\begin{align}
\label{paq4}
\partial_T n^{(1)}+2n^{(1)}\partial_
Xn^{(1)}+\frac{1}{a^2}\partial_X\Delta
n^{(1)}&=-\partial_X(v_x^{(2)}-n^{(2)})\\
\label{paq5}
\partial_X(v_x^{(2)}-\phi^{(2)})&=\partial_Tn^{(1)}+n^{(1)}\partial_
Xn^{(1)}\\
\label{paq6}
\partial_y\phi^{(2)}&=\frac{1}{a^2}\partial^3_{XXy}n^{(1)}\\
\label{paq7}
\partial_z\phi^{(2)}&=\frac{1}{a^2}\partial^3_{XXz}n^{(1)}
\end{align}
while the cancellation of $r_2$ is equivalent to $\phi^{(1)}=n^{(1)}$
which has already been imposed.
\subsubsection{Cancellation of terms of order  $\e^{3/2}$}\label{secteps32}
It is possible to cancel the terms of order $\e^{3/2}$ for the
equation on the density and on the first component of the velocity;
the fact that $N^3=R^1_3=0$ is actually a direct consequence of
\eqref{paq0}-\eqref{paq2} and\eqref{paq3}. Looking for the
cancellation of the other components of the velocity equation, namely,
setting $R^2_3=R^3_3=0$ yields respectively
\begin{align*}
%\partial_y v_y^{(3)}+\partial_z v_z^{(3)}&=0 \;(\text{from the equation of}\;n^{\e})\\
0&=-\frac{1}{a}\lbrack \partial^2_{zT}n^{(1)}+n^{(1)}\partial^2_{zX}n^{(1)} \rbrack ,\\
0&=\frac{1}{a}\lbrack \partial^2_{yT}n^{(1)}+n^{(1)}\partial^2_{yX}n^{(1)} \rbrack\ ,
\end{align*}
which are inconsistent with the other equations on $n^{(1)}$;
consequently, {\it we cannot expect a better error than $O(\e^{3/2})$
  on the equations for the transverse components of the velocity}.

\subsubsection{Cancellation of terms of order  $\e^2$}\label{secteps2}

In order to justify the ZK approximation, we need to cancel the
$O(\eps^2)$ terms in the equation for $\phi^\eps$, that is, to impose
$r^4=0$, leading to the equation  
\begin{align}
\label{paq9}
\phi^{(2)}-n^{(2)}=\Delta n^{(1)}-\frac{1}{2}(n^{(1)})^2.
\end{align}

\subsubsection{Derivation of the Zakharov-Kuznetsov equation}
Combining (\ref{paq4}), (\ref{paq5}) and (\ref{paq9}), we find that
$n^{(1)}$ must solve the Zakharov-Kuznetsov equation,
\begin{equation}\label{ZKapp}
2\partial_T n^{(1)}+2n^{(1)}\partial_Xn^{(1)}+(\frac{1}{a^2}+1)\Delta \partial_X n^{(1)}=0.
\end{equation}

\subsubsection{Construction of the profiles}
The ZK equation being locally well posed on $H^s(\R^d)$, for all
$s>d/2+1$, \footnote{ See \cite{Fa, LiS, RV} for the Cauchy theory  in larger spaces.} we can consider a solution $n^{(1)}\in C(\lbrack
0,T_0\rbrack; H^s(\R^d)),$ $s \geq5$ to \eqref{ZKapp}
for some $T_0>0$. We show here how to construct all the quantities
involved in \eqref{solapp} in terms of $n^{(1)}$.
\begin{itemize}
\item In agreement with \eqref{paq0}, we set
  $\phi^{(1)}=v_x^{(1)}=n^{(1)}$.
\item Equations \eqref{paq1}-\eqref{paq2} then give $v_y^{(1)},
  v_z^{(1)}\in C(\lbrack 0,T_0\rbrack; H^{s-1}(\R^d))$.
\item We then use  \eqref{paq3} to obtain $v_y^{(2)},v_z^{(2)} \in
  C(\lbrack 0,T_0\rbrack; H^{s-2}(\R^d))$.
\item Taking $\phi^{(2)}=\frac{1}{a^2}\partial_{XX}^2 n^{(1)}\in
  C(\lbrack 0,T_0\rbrack; H^{s-2}(\R^d))$ then
  ensures that \eqref{paq6} and \eqref{paq7} hold.
\item We then get the density corrector $n^{(2)}\in C(\lbrack
  0,T_0\rbrack; H^{s-2}(\R^d))$ by \eqref{paq9}.
\item We recover $v_x^{(2)}\in C(\lbrack
  0,T_0\rbrack; H^{s-2}(\R^d))$ from \eqref{paq4} or \eqref{paq5} --- the
  fact that $n^{(1)}$ solves the ZK equation ensures that we find the
  same expression with \eqref{paq4} or \eqref{paq5} .
\end{itemize}

\medbreak

The computations above and the explicit expression of the remaining
residual terms in \eqref{2}, \eqref{3}, \eqref{4}, \eqref{5} and
\eqref{6} imply the following consistency result.
\begin{proposition}\label{propconsist}
Let $T_0>0$, $n_0\in H^s(\R^d)$ ($s\geq 5$) and $n^{(1)}\in C(\lbrack 0,
T_0\rbrack;H^s(\R^d))$, solving
$$2\partial_T
n^{(1)}+2n^{(1)}\partial_Xn^{(1)}+\left(1+\frac{1}{a^2}\right)\Delta \partial_Xn^{(1)}=0,\qquad n^{(1)}_{\vert_{t=0}}=n_0.$$
Constructing the other profiles as indicated above, the approximate
solution 
 $(n^\eps,\bv^\eps,\phi^\eps)$ given by \eqref{solapp} solves
 (\ref{scaEP}) up to order $\eps^3$ in $\phi^\eps$, $\eps^2$ in $n^\eps$, $v_x^\eps$,
 and up to order $\eps^{3/2}$ in $v_y^\eps,$ $v_z^\eps$:
\begin{equation}
    \label{scaEPquar}
    \left\lbrace
   \begin{array}{lcl}
   	\dsp\partial_tn^\e+\nabla\cdot ((1+\eps n^\e){\bf v}^\e)
        =\e^{2} N^\e,\\
	\dsp\partial_t{\bf v}^\e+\eps ({\bf v}^\e\cdot \nabla){\bf
          v}^\e+\nabla\phi^\e+a\eps^{-1/2}{\bf e}\wedge {\bf
          v}^\e=\eps^{3/2} R^\e,\\
	\dsp -\epsilon
        ^2\Delta\phi^\e+e^{\epsilon\phi^\e}-1=\eps n^\e+ \epsilon^3 r^\e,
    \end{array}\right.
\end{equation}
with $R^\eps=(\eps^{1/2} R^\e_1,R^\e_2,R^\e_3)$ and
$$
\abs{N^\e}_{L^\infty([0,\frac{T_0}{\eps}];H^{s-5})} +\sum_{j=1}^3\abs{R^\e_j}_{L^\infty([0,\frac{T_0}{\eps}];H^{s-5})}  +\abs{r^\eps}_{L^\infty([0,\frac{T_0}{\eps}];H^{s-4})}\leq C\big(T_0,\abs{n_0}_{H^s}\big).
$$
\end{proposition}

\subsection{Justification of the Zakharov-Kuznetsov approximation}

We are now set to justify the Zakharov-Kuznetsov approximation.
\begin{theorem}\label{thmjustif}
Let $n^0\in H^{s+5}$, with $s>d/2+1$, such that $1+ n^0\geq c_0$
on $\R^d$ for some constant $c_0>0$. There exists $T_1>0$ such that
for all $\eps\in (0,1)$,
\item{\bf i.} The Zakharov-Kuznetsov approximation $(n^\e,\bv^\e,\phi^\e)$ of Proposition
  \ref{propconsist} exists on the time interval $[0,T_1/\eps]$;
\item{\bf ii.} There exists a unique solution $(\underline{n},\underline{\bf v},\underline{\phi})\in C(\lbrack
 0,\frac{{T_1}}{\eps}\rbrack; H^{s}(\R^d)\times
 H^{s+1}_\eps(\R^d)^d\times H^{s+1}(\R^d))$ provided by Theorem
 \ref{thmIVPunif} to the Euler-Poisson equations \eqref{scaEP} with
 initial condition $(\underline{n}^0,\underline{\bf
   v}^0,\underline{\phi}^0)=(n^\e,\bv^\e,\phi^\e)_{\vert_{t=0}}$.
\item Moreover, one has the error estimate
$$
\forall 0\leq t\leq T_1/\eps, \qquad \abs{\underline{n}(t)-n^\e(t)}_{H^s}^2+\abs{
  \underline{\bv}(t)-\bv^\e(t)}_{H^{s+1}_\eps}^2\leq \eps^{3/2} t C(\frac{1}{c_0},T_1,\abs{n^0}_{H^{s+5}}).
$$
\end{theorem}
\begin{remark}\label{rmm1}
The error of the approximation is $O(\eps^{3/2})$ for times of order
$O(1)$ but of size $O(\eps^{1/2})$ for times of order
$O(1/\eps)$. Looking at $\eqref{solapp}$, this is a relative error of
size $O(\eps^{1/2})$ for $n$ and $v_x$ but of size $O(1)$ for $v_y$
and $v_z$. Consequently, the
Zakharov-Kuznetsov approximation provides a good approximation for
the density and the longitudinal velocity, but not, for large times,
for the transverse velocity (at least, we did not prove it in Theorem \ref{thmjustif}).
\end{remark}
\begin{remark}\label{rmm2}
In the one dimensional case (KdV approximation), all the terms of
order $\eps^{3/2}$ can be cancelled (see \S \ref{secteps32}) and the
residual in the second equation of \eqref{scaEPquar} is of size
$O(\eps^2)$ instead of $O(\eps^{3/2})$. The error in the theorem then
becomes $O(\eps^2 t)$, which gives a relative error of size $O(\eps)$
for both the density and the velocity over large times
$O(1/\eps)$. In the one dimensional case, it is also possible to
construct higher order approximations by including the order three and
higher terms in the ansatz \eqref{ZKansatz}. This has been done in \cite{GuoPu}.
\end{remark}
\begin{remark}
In \cite{GuoPu}, the authors justify the KdV approximation
(corresponding to the one
dimensional ZK approximation) by looking at an exact solution as a perturbation of
the approximate solution, $(n_{ex},v_{ex})=(n_{app},v_{app})+\eps^k (n_R,v_R)$,
(with $k>0$ depending on the order of the approximation). They study
the equations satisfied by $(n_R,v_R)$, which requires subtle
estimates. Our approach is much simpler: we prove uniform (with
respect to $\eps$) well posedness of the Euler Poisson equation, from
which we deduce very easily that any consistent approximation remains
close to the exact solution (but of course, in a lower norm).
\end{remark}
\begin{proof}
Let us take $0<T_1\leq \min\{\underline{T},T_0\}$, where
$\underline{T}/\eps$ is the existence time of the exact solution
provided by Theorem \ref{thmIVPunif}, and $T_0/\eps$ the existence
time of the approximate solution in Proposition \ref{propconsist}. Denote by $(\underline{n},\underline{\bv},\underline{\phi}) \in C(\lbrack
 0,\frac{{T_1}}{\eps}\rbrack; H^{s}(\R^d)\times
 H^{s+1}_\eps(\R^d)^d\times H^{s+1}(\R^d)$ the
exact solution to \eqref{scaEP} with the same initial conditions as
$(n^\eps,\bv^\eps,\phi^\eps)$ furnished by Theorem
\ref{thmIVPunif}. We also write
$$
(n,\bv,\phi)=(n^\e,\bv^\e,\phi^\e)-(\underline{n},\underline{v},\underline{\phi}).
$$
Taking the difference between (\ref{scaEPquar}) and \eqref{scaEP}, we
get
\begin{equation}
    \label{EPlinquar}
    \left\lbrace
   \begin{array}{lcl}
   	\dsp\partial_tn+(1+\eps\underline{n})\nabla\cdot{\bf
          v}+\eps\underline{\bv}\cdot \nabla n +\eps \nabla n^\e\cdot\bv+\eps
        (\nabla\cdot \bv^\e) n =\eps f,\\
	\dsp\partial_t{\bf v}+\eps (\underline{\bf v}\cdot \nabla){\bf v}+\eps\bv\cdot\nabla\bv^\e+\nabla\phi+a\eps^{-1/2}{\bf e}\wedge {\bf v}=\eps {\bf g},\\
	\dsp \Meb \nabla\phi=\nabla n-\eps \phi e^{\eps \phi^\e} \nabla
\phi^\e+\eps {\bf h}.
    \end{array}\right.
\end{equation}
with $f=\eps N^\e$, ${\bf g}=\eps^{1/2}R^\e$ and $\dsp {\bf h}=\eps \nabla
r^\eps+\big(\frac{e^{\eps\underline{\phi}}-e^{\eps\phi^\e}}{\e}-(\underline{\phi}-\phi^\e) e^{\eps\phi^\e}\big)\nabla
\phi^\e$. This system is of the form \eqref{EPlin} with additional
linear terms (namely, $\eps \nabla n^\e\cdot\bv+\eps
        (\nabla\cdot \bv^\e) n$ in the first equation,  $\eps
        \bv\cdot\nabla\bv^\e$ in the second one, and $-\eps \phi e^{\eps \phi^\e} \nabla
\phi^\e$ in the third one) that do not affect the derivation of the
energy estimate \eqref{ests}. The only difference is that the constant
$C_s$ in \eqref{ests} must also depend on
$\abs{(n^\e,\bv^\e,\phi^\e)}_{H^{s+1}_T} $. Since the initial conditions for
\eqref{EPlinquar} are identically zero, this yields, for all $0\leq
T\leq T_1/\eps$,
\begin{eqnarray}
\nonumber
\sup_{[0,T]}\big(\abs{n}_{H^s}^2+\abs{ \bv}_{H^{s+1}_\eps}^2\big)& \leq& 
\exp(\eps \widetilde C_s T) \\
\label{estsquar}
& &\times \eps T\big(\abs{f}^2_{H^s_T}+\abs{{\bf g}}^2_{{H}^{s+1}_{\eps,T}}+\abs{ {\bf h}}_{H^s_T}^2+\abs{n}_{H^s}^2+\abs{ \bv}_{H^{s+1}_\eps}^2\big),
\end{eqnarray}
with  $\widetilde C_s=
C\big(\frac{1}{c_0},\abs{\un}_{H^s_T},\abs{\ubv}_{H^{s+1}_{\eps,T}},\abs{\underline{\phi}}_{H^s_T},\abs{\dt
  ( \underline{n}, \underline{\bv},
  \underline{\phi})}_{L^\infty_T},\abs{(n^\e,\bv^\e,\phi^\e)}_{H^{s+1}_T}\big)$. Taking
if necessary a smaller $T_1>0$, this implies that for all $ 0\leq t\leq T_1/\eps$,
\begin{equation}\label{estsquarbis}
\abs{n(t)}_{H^s}^2+\abs{ \bv(t)}_{H^{s+1}_\eps}^2\leq
\eps t \exp(\widetilde C_s T_1) 
\times\big(\abs{f}^2_{H^s_{t}}+\abs{{\bf g}}^2_{{H}^{s+1}_{\eps,t}}+\abs{ {\bf h}}_{H^s_{t}}^2\big).
\end{equation}
Now, one the one hand, we have
\begin{eqnarray*}
\widetilde C_s&=&
C\big(\frac{1}{c_0},\abs{\un}_{H^s_T},\abs{\ubv}_{H^{s+1}_{\eps,T}},\abs{\underline{\phi}}_{H^s_T},\abs{\dt
  ( \underline{n}, \underline{\bv},
  \underline{\phi})}_{L^\infty_T},\abs{(n^\e,\bv^\e,\phi^\e)}_{H^{s+1}_T}\big)\\
&=& C\big(\frac{1}{n_0},\abs{n^0}_{H^{s+5}}\big),
\end{eqnarray*}
where we used Theorem \ref{thmIVPunif} to control the norms of
$(\underline{n},\underline{\bv},\underline{\phi})$ in terms of
$(n^0,\bv^0)$ (with $\bv^0$ given in terms of $n^0$ by
$\bv^0=\bv^\eps_{\vert_{t=0}}$), and the expression of all the components of
$(n^\eps,\bv^\e,\phi^\eps)$ in terms of $n^{(1)}$ to control the norms
of the approximate solution. On the other hand, we get from the
definition of $f$, ${\bf g}$, ${\bf h}$ and
Proposition \ref{propconsist} that
$$
\abs{f}^2_{H^s_{t}}+\abs{{\bf g}}^2_{{H}^{s+1}_{\eps,t}}+\abs{ {\bf
    h}}_{H^s_{t}}^2\leq \eps^{1/2} C(T_1,\abs{n^0}_{H^{s+5}}).
$$
We deduce therefore from \eqref{estsquarbis} that
$$
\forall 0\leq t\leq T_1/\eps, \qquad \abs{n(t)}_{H^s}^2+\abs{
  \bv(t)}_{H^{s+1}_\eps}^2\leq \eps^{3/2}t C(\frac{1}{c_0},T_1,\abs{n^0}_{H^{s+5}}).
$$
\end{proof}
\section{The Euler-Poisson system with isothermal pressure}\label{ISP}

In \cite{GuoPu}, the authors derived and justified a version of 
the KdV (and therefore $d=1$) equation in the case where the
isothermal pressure is not
neglected. In the general of dimension $d\geq 1$, the equations
\eqref{eulerpoisson} are then given by 
\begin{equation}
    \label{eulerpoissongen}
    \left\lbrace
   \begin{array}{lcl}
   	\dsp\partial_tn+\nabla \cdot {\bf v}+\nabla\cdot(n{\bf v})=0,\\
	\dsp\partial_t{\bf v}+({\bf v}\cdot\nabla){\bf
          v}+\nabla\phi+\alpha \frac{\nabla n}{1+n}+a{\bf e}\wedge {\bf v}=0,\\
	\dsp\Delta\phi-e^{\phi}+1+n=0,
    \end{array}\right.
\end{equation}
where $\alpha$ is a positive constant related to the ratio of the ion
temperature over the ion mass. In the case of cold plasmas considered
in the previous sections, one has $\alpha=0$. In \cite{GuoPu}, the
cases $\alpha=0$ and $\alpha>0$ are treated differently, and the limit
$\alpha\to 0$ (or, for instance, $\alpha=O(\eps)$) cannot be handled. We
show here that the proof of Theorem \ref{thmIVPunif} can easily be adapted to
the general case $\alpha\geq 0$, hereby allowing the limit $\alpha\to 0$
and providing a generalization of the results of \cite{GuoPu} to the
case $d\geq 1$.
We first extend Theorem \ref{thmIVPunif} to the general
Euler-Poisson system with isothermal pressure
(\ref{eulerpoissongen}). We then indicate how to derive and justify a
generalization of the Zakharov-Kuznetsov approximation taking into
account this new term, in the same spirit as the KdV approximation
derived in the one-dimensional case in \cite{GuoPu}.

\subsection{The Cauchy problem for the Euler-Poisson system with
  isothermal pressure}

As in Section \ref{LWEP}, we work with rescaled equations. More
precisely, we perform the same rescaling as for \eqref{scaEP}; without
the ``cold plasma `` assumption, this system must be replaced by
\begin{equation}
    \label{scaEPgen}
    \left\lbrace
   \begin{array}{lcl}
   	\dsp\partial_tn+\nabla\cdot ((1+\eps n){\bf v}) =0,\\
	\dsp\partial_t{\bf v}+\eps ({\bf v}\cdot \nabla){\bf
          v}+\nabla\phi+\alpha \frac{\nabla n}{1+\eps n}+a\eps^{-1/2}{\bf e}\wedge {\bf v}=0,\\
	\dsp -\epsilon ^2\Delta\phi+e^{\epsilon\phi}-1=\epsilon n,
    \end{array}\right.
\end{equation}
with $\alpha \geq 0$. The presence of the extra term $\alpha
\frac{\nabla n}{1+\eps n}$ in the second equation induces a smoothing
effect that allows the authors of \cite{GuoPu} to use the
pseudo-differential estimates of Grenier \cite{Grenier}. However,
these smoothing effects disappear when $\alpha\to 0$, and the
existence time thus obtained is not uniform with respect to
$\alpha$. We provide here a generalization of Theorem \ref{thmIVPunif}
that gives a uniform existence time with respect to $\eps$ {\it and}
$\alpha$ (so that solutions to (\ref{scaEPgen}) provided by Theorem
\ref{thmIVPunifgen} can be seen as limits when $\alpha\to0$ of
solutions to (\ref{scaEPgen})). This evanescent smoothing effect is
taken into account by working with $n\in H^{s+1}_{\eps\alpha}(\R^d)$
rather than $n\in H^s(\R^d)$ as in Theorem \ref{thmIVPunif} (from the
definition \eqref{defHseps} of $H^s_{\eps\alpha}$,
$H^{s+1}_{\eps\alpha}(\R^d)$ coincides with $H^s(\R^d)$ when $\alpha=0$).
\begin{theorem}\label{thmIVPunifgen}
Let $s>\frac{d}{2}+1$, $\alpha_0>0$ and $n_0\in H^{s+1}_{\eps\alpha_0}(\R^d)$, ${\bf v}_0 \in H^{s+1}(\R^d)^d$ such that $1-\abs{n_0}_\infty\geq c_0$ for some $c_0>0$. \\
Then there exist $\underline{T}>0$ such that for all $\eps\in (0,1)$
and $\alpha\in (0,\alpha_0)$, there is 
 a unique solution $(n^{\eps,\alpha},{\bf v}^{\eps,\alpha})\in C(\lbrack 0,\frac{\underline{T}}{\eps}\rbrack; H^{s+1}_{\eps\alpha}(\R^d)\times H^{s+1}_\eps(\R^d)^d)$ of \eqref{scaEP} such that $1+\eps n>c_0/2$ and $\phi^{\eps,\alpha} \in
C(\lbrack 0,T\rbrack; H^{s+1}(\R^d))$. \\
Moreover the family $(n^{\eps,\alpha},\bv^{\eps,\alpha},\nabla\phi^{\eps,\alpha})_{\eps\in (0,1),\alpha\in(0,\alpha_0)}$ is uniformly bounded in $H^{s+1}_{\eps\alpha}\times H^{s+1}_\eps\times H^{s-1}$.
\end{theorem}
\begin{proof}
The proof follows the same steps as the proof of Theorem
\ref{thmIVPunif}; in addition to the operator $\Me$ defined in
\eqref{linearize}, we also need to define another second order
self-adjoint operator $\Nea$ as
\begin{equation}\label{defNea}
\Nea=\frac{1}{1+n}+\frac{1}{1+n}\Me\frac{1}{1+n},
\end{equation}
provided that $\inf_{\R^d}(1+n)>0$.

\noindent
{\bf Step 1. Preliminary results.} The operator $\Nea$ defined in
\eqref{defNea} inherit from the properties of $\Me$ the following
estimates that echoe \eqref{equiv},
$$
\begin{array}{l}
\dsp \big( u,\Nea v\big)\leq
C\big(\abs{\phi}_\infty,\abs{n}_{\infty}\big)\babs{\frac{u}{1+\eps
    n}}_{H^1_{\eps\alpha}}\babs{\frac{v}{1+\eps n}}_{H^1_{\eps\alpha}},
\\
\dsp \babs{\frac{u}{1+\eps n}}_{H^1_{\eps\alpha}}^2 \leq C\big(\abs{\phi}_\infty,\abs{n}_{\infty}\big) \big( u,\Nea u\big).
\end{array}
$$
We also have the following commutator estimates, that are similar to
those satisfied by $\Me$ (see Step 1 in the proof of Theorem \ref{thmIVPunif}),
$$
\begin{array}{l}
\dsp \big( u,[\dt,\Nea]u\big)\leq \eps C\big(\frac{1}{c_0},\abs{\dt
  n}_\infty,\abs{n}_{W^{1,\infty}},\abs{\dt \phi}_\infty,\abs{\phi}_\infty\big)\babs{\frac{u}{1+\eps
    n}}_{H^1_{\eps \alpha}}^2\\
\dsp \big(u,[f\partial_j,\Nea]u\big) \leq
C\big(\frac{1}{c_0},\abs{(n,\phi,f)}_{W^{1,\infty}},\sqrt{\eps\alpha}\abs{\nabla\partial_j
n}_\infty,\sqrt{\eps}\abs{\nabla\partial_j
f}_\infty\big)\\
\hspace{4cm}\dsp \times \babs{\frac{u}{1+\eps
  n}}_{H^1_{\eps\alpha}}^2.
\end{array}
$$
\noindent
{\bf Step 2.} $L^2$ estimates for a linearized system. Without the
``cold plasma'' approximation, one must replace \eqref{EPlin} by 
\begin{equation}    
\label{EPlingen}
    \left\lbrace
   \begin{array}{lcl}
   	\dsp\partial_tn+(1+\eps\underline{n})\nabla\cdot{\bf v}+\eps\underline{\bv}\cdot \nabla n =\eps f,\\
	\dsp\partial_t{\bf v}+\eps (\underline{\bf v}\cdot \nabla){\bf v}+\nabla\phi +\alpha \frac{\nabla n}{1+\eps \underline{n}}+a\eps^{-1/2}{\bf e}\wedge {\bf v}=\eps {\bf g},\\
	\dsp \Meb \nabla\phi=\nabla n+\eps {\bf h}.
    \end{array}\right.
\end{equation}
The presence of the new term in the second equation of
(\ref{EPlingen}) yields some smoothing effect on the estimate on $n$
that are absent when $\alpha=0$;
this smoothing is measured with the $H^1_{\eps\alpha}(\R^d)$ norm,
which coincides with the $L^2$ norm used for (\ref{EPlingen}) when
$\alpha=0$. More precisely, we want to prove here that 
\begin{eqnarray}
\nonumber
\sup_{[0,T]}\big(\abs{n}_{H^1_{\eps\alpha}}^2+\abs{\bv}_{H^1_\eps}^2\big) \leq \exp(\eps C_0 T)\\
\label{EstL2gen}
\times \big(\abs{n_{\vert_{t=0}}}_{H^1_{\eps\alpha}}^2+ \abs{\bv_{\vert_{t=0}}}_{H^1_\eps}^2\big)+\eps T(\abs{f}^2_{H^1_{\eps\alpha,T}}+\abs{{\bf g}}^2_{H^1_{\eps,T}}+\abs{{\bf h}}_{L^2_T}^2)\big),
\end{eqnarray}
with
$C_0=C(\frac{1}{c_0},\abs{(\underline{n},\underline{\bv},\phi)}_{W^{1,\infty}_T},\abs{\dt
  ( \underline{n}, \underline{\bv},
  \underline{\phi})}_{L^\infty_T},\sqrt{\eps}\abs{ \ubv}_{W^{2,\infty}_T},\sqrt{\eps\alpha}\abs{ n}_{W^{2,\infty}_T})$.\\
Instead of multiplying the first equation of \eqref{EPlingen} by $(1+\eps
\underline{n})^{-1}$ as in the case $\alpha=0$, we multiply it by
$\Neab$ to obtain
\begin{eqnarray*}
\big(\Neab\dt
n,n\big)+\big([1+\alpha\frac{1}{1+\eps\underline{n}}\Meb]\nabla\cdot
\bv,n\big)+\eps\big(\Neab\bv\cdot\nabla n,n)\\
=\eps \big(\Neab f,n\big),
\end{eqnarray*}
which can be rewritten under the form
\begin{eqnarray}
\nonumber
\frac{1}{2}\dt \big(\Neab n,n\big)-\frac{1}{2}\big([\dt,\Neab]n,n\big)+
\big([1+\alpha\frac{1}{1+\eps\underline{n}}\Meb]\nabla\cdot
\bv,n\big)\\
 \label{estL21gen}
-\eps\frac{1}{2}\big(n,[\ubv\cdot\nabla,\Neab]n\big)
=\eps \big(\Neab f,n\big).
\end{eqnarray}
As in the proof of Theorem \ref{thmIVPunif}, we take the $L^2$ scalar product
of the second equation with $\Me \bv$; after remarking that
$$
\Meb\big(\nabla\phi+\alpha\frac{\nabla n}{1+\eps\underline{n}}\big)=
[1+\alpha \Meb \frac{1}{1+\eps\underline{n}}]\nabla
n+\eps {\mathbf h}
$$
we obtain with the same computations the following generalization of \eqref{estL22},
\begin{eqnarray}
\nonumber
\lefteqn{\frac{1}{2}\dt \big(\Meb\bv,\bv\big)-\frac{1}{2}\eps\big(\dt \underline{\phi}e^{\eps\underline{\phi}}\bv,\bv\big)
-\frac{1}{2}\eps\big(\bv,(\nabla\cdot\ubv)\Meb \bv\big)}\\
\nonumber
& &-\frac{1}{2}\eps
\big(\bv,[\ubv\cdot\nabla,\Meb]\bv\big)+\big([1+\alpha \Meb
\frac{1}{1+\eps\underline{n}}]\nabla n,\bv\big)\\
& &
\label{estL22gen}
=-\eps\big({\bf h},\bv\big)+\eps \big(\Meb{\bf g},\bv\big).
\end{eqnarray}
Adding (\ref{estL22gen}) to (\ref{estL21gen}), and proceeding exactly
as in the proof of Theorem \ref{thmIVPunif}, we get \eqref{EstL2gen}.

\noindent
{\bf Step 3. $H^s$ estimates for a linearized system.} We want to prove here that for all $s\geq 0$, 
the solution $(n,\bv,\phi)$ to (\ref{EPlingen}) satisfies, for all $s\geq t_0+1$,
\begin{eqnarray}
\nonumber
\sup_{[0,T]}\big(\abs{n}_{H^{s+1}_{\eps\alpha}}^2+\abs{ \bv}_{H^{s+1}_\eps}^2\big)& \leq& 
\exp(\eps C_s T) \times \Big(\abs{n_{\vert_{t=0}}}_{H^{s+1}_{\eps\alpha}}^2+ \abs{ \bv_{\vert_{t=0}}}_{H^{s+1}_\eps}^2\big)\\
\label{estsgen}
& &\!\!\!\!\!\!+\eps T(\abs{f}^2_{H^{s+1}_{\eps\alpha,T}}+\abs{{\bf g}}^2_{{H}^{s+1}_{\eps,T}}+\abs{ {\bf h}}_{H^s_T}^2+\abs{n}_{H^{s+1}_{\eps\alpha}}^2+\abs{ \bv}_{H^{s+1}_\eps}^2)\Big),
\end{eqnarray}
with $C_s= C(\frac{1}{c_0},\abs{\un}_{H^{s+1}_{\eps\alpha,T}},\abs{\ubv}_{H^{s+1}_{\eps,T}},\abs{\underline{\phi}}_{H^s_T},\abs{\dt ( \underline{n}, \underline{\bv}, \underline{\phi})}_{H^s_T})$.\\
Applying $\Lambda^s$ to the three equations of (\ref{EPlin}), and
writing $\tilde n=\Lambda^s n$, $\tilde \bv=\Lambda^s \bv$, and
$\tilde\phi=\Lambda^s \phi$, we get
\begin{equation}
    \label{EPlinHsgen}
    \left\lbrace
   \begin{array}{lcl}
   	\dsp\partial_t\tilde n+(1+\eps \underline{n})\nabla\cdot \tilde{\bf v} +\eps\ubv\cdot\nabla\tilde n=\eps\tilde f,\\
	\dsp\partial_t\tilde{\bf v}+\eps (\underline{\bf v}\cdot
        \nabla)\tilde{\bf v}+\nabla\tilde\phi+\alpha \frac{\nabla
          \tilde n}{1+\eps \underline{n}} +a\eps^{-1/2}{\bf e}\wedge \tilde {\bf v}=\eps \tilde {\bf g},\\
	\dsp \Meb\nabla\tilde\phi=\nabla\tilde n+\eps\tilde {\bf h}.
    \end{array}\right.
\end{equation}
with $\tilde f$ and $\tilde{\bf h}$ as in the proof of Theorem
\ref{thmIVPunif}, while $\tilde{\bf g}$ must be changed into $\tilde{\bf
  g}=\tilde{\bf g}+\alpha[\Lambda^s,\frac{\underline{n}}{1+\eps
  \underline{n}}]\nabla\tilde n$. Using Step 2, one can mimic the proof of
the Theorem \ref{thmIVPunif}. The only new ingredients needed are a control in
$H^{1}_{\eps\alpha}(\R^d)$ of $\alpha[\Lambda^s,\frac{\underline{n}}{1+\eps
  \underline{n}}]\nabla\tilde n $ (the new term in $\tilde{\bf g}$) and a
control of $\tilde f$ in $H^1_{\eps\alpha}(\R^d)$ instead of
$L^2(\R^d)$. Classical commutator estimates (see for instance
\cite{Lcomm}) yield for $s>d/2+1$,
$$
\begin{array}{l}
\dsp \babs{\alpha[\Lambda^s,\frac{\underline{n}}{1+\eps
  \underline{n}}]\nabla\tilde n}_{H^1_\eps}\leq
C(\frac{1}{c_0},\abs{\underline{n}}_{H^{s+1}_{\eps\alpha}})\abs{\tilde
  n}_{H^{s+1}_{\eps\alpha}},\\
\dsp \abs{\tilde f}_{H^1_{\eps\alpha}}\leq \abs{f}_{H^{s+1}_{\eps\alpha}}+
\big(\abs{\un}_{H^{s+1}_{\eps\alpha}}+\abs{\ubv}_{H^{s+1}_\eps}\big)\times
\big(\abs{\bv}_{H^{s+1}_\eps}+\abs{n}_{H^{s+1}_{\eps\alpha}}\big),
\end{array}
$$
so that (\ref{estsgen}) follows exactly as (\ref{ests}).  

\noindent
{\bf Step 4. End of the proof.} The end of the proof is exactly
similar the same as for Theorem \ref{thmIVPunif} (the exact solution is no
longer furnished by Theorem \ref{ivp} but by a standard iterative
scheme).
\end{proof}
\begin{remark}
The proof given in Theorem \ref{ivp} for the case $\alpha =0$ does not work when $\alpha>0.$
\end{remark}

\subsection{Derivation of a  Zakharov-Kuznetsov
equation in presence of isothermal pressure}

We proceed similarly to the cold plasma case but we replace  the
ansatz \eqref{solapp} by
\begin{equation}
\label{ZKansatzbis}
\begin{split}
n^{\e}&=\,n^{(1)}(x-ct, y,z,\e t) +\e n^{(2)}\\
\phi^{\e}&=\,\phi^{(1)}(x-ct, y,  z, \e t) +\e\phi^{(2)}\\
v_x^{\e}&=\,v_x^{(1)}(x-ct, y, z, \e t) +\e v_x^{(2)}\\
v_y^{\e}&=\e^{1/2}\,v_y^{(1)}(x-ct, y,  z,\e t) +\e v_y^{(2)}\\
v_z^{\e}&=\e^{1/2}\,v_z^{(1)}(x-ct, y, z, \e t) +\e v_z^{(2)},
\end{split}
\end{equation}
where the velocity $c$ has to be determined.\\
Following the strategy of \S \ref{ZKA}, we plug this ansatz into
\eqref{scaEPgen}, and choose the profiles in \eqref{ZKansatzbis} in
order to cancel the leading order terms.

\subsubsection{Cancellation of terms of order $\eps^0$}  

Canceling the leading order $O(1)$ yields
\begin{equation}
\label{zero}
\begin{split}
-c\partial_Xn^{(1)}+\partial _Xv_x^{(1)}=0, \\
-c\partial_Xv_x^{(1)}+\partial_X \phi^{(1)}+\alpha\partial_X n^{(1)}=0,\\
(1+\alpha)\partial_y n^{(1)}-av_z^{(1)}=0,\\
(1+\alpha)\partial_z n^{(1)}+av_y^{(1)}=0.
\end{split}
\end{equation}

\subsubsection{Cancellation of terms of order $\eps^{1/2}$}  

We get at this step
\begin{equation}
\label{poq4}
v_y^{(2)}=\frac{(1+\alpha)^{3/2}}{a^2}\partial_{Xy}^2 n^{(1)},\qquad 
v_z^{(2)}=\frac{(1+\alpha)^{3/2}}{a^2}\partial_{Xz}^2 n^{(1)}.
\end{equation}

\subsubsection{Cancellation of terms of order $\eps$}

Proceeding as in \S \ref{secteps1}, we get 
\begin{align}
\label{poq1}
\partial_T n^{(1)}+2n^{(1)}\partial_
Xn^{(1)}+\frac{1}{a^2}\partial_X\Delta
n^{(1)}&=-\partial_X(v_x^{(2)}-cn^{(2)})\\
\label{poq2}
\partial_X(cv_x^{(2)}-\alpha
n^{(2)}-\phi^{(2)})&=c\partial_Tn^{(1)}+n^{(1)}\partial_ Xn^{(1)}\\
\label{pouet1}
\partial_y\phi^{(2)}&=\frac{(1+\alpha)^{3/2}}{a^2}\partial^3_{XXy}n^{(1)}\\
\label{pouet2}
\partial_z\phi^{(2)}&=\frac{(1+\alpha)^{3/2}}{a^2}\partial^3_{XXz}n^{(1)}\\
\label{p1n1}
\phi^{(1)}&=n^{(1)}
\end{align}
After replacing $\phi^{(1)}$ by $n^{(1)}$ according to \eqref{p1n1},
one readily checks that the first two equations of \eqref{zero} are
consistent if and only if $c=\sqrt{1+\alpha}$.

\subsubsection{Cancellation of terms of order $\eps^{3/2}$}

As in \S \eqref{secteps32}, the cancellation of the $O(\eps^{3/2})$
terms for the density and the
longitudinal velocity equations is automatic, but it is not possible
for the equations on the transverse velocity.

\subsubsection{Cancellation of terms of order $\eps^{2}$}

As in \S \ref{secteps2} we only need to cancel the $O(\eps^2)$ terms
in the equation for $\phi^\e$, which yields here
\begin{align}\label{poq3}
\phi^{(2)}-n^{(2)}=\Delta n^{(1)}-\frac{1}{2}(n^{(1)})^2.
\end{align}

\subsubsection{Derivation of the  Zakharov-Kuznetsov equation}
Combining (\ref{poq1}), (\ref{poq2}) and (\ref{poq3}), we find that
$n^{(1)}$ must solve the following  Zakharov-Kuznetsov
equation (which coincides with \eqref {ZKapp} when $\alpha =0$),
\begin{equation}\label{newZK}
2c\partial_T n^{(1)}+2cn^{(1)}\partial_X n^{(1)}+(1+\frac{c}{a^2})\Delta \partial_X n^{(1)}=0.
\end{equation}

\subsubsection{Construction of the profiles}

The profiles
involved in \eqref{ZKansatzbis} are constructed in terms of $n^{(1)}$
as follows:
\begin{itemize}
\item In agreement with \eqref{p1n1} and the first two equations of \eqref{zero}, we set
  $\phi^{(1)}=n^{(1)}$ and $v_x^{(1)}=c n^{(1)}$, with $c=\sqrt{1+\alpha}$.
\item The last two equations of  \eqref{zero} then give $v_y^{(1)},
  v_z^{(1)}$.
\item We then use  \eqref{poq4} to obtain $v_y^{(2)},v_z^{(2)}$.
\item We take $\phi^{(2)}=\frac{(1+\alpha)^{3/2}}{a^2}\partial_{XX}^2
  n^{(1)}$ to satisfy
 \eqref{pouet1} and \eqref{pouet2}.
\item We then get the density corrector $n^{(2)}$ by \eqref{poq3}.
\item We recover $v_x^{(2)}$ from \eqref{poq1} or \eqref{poq2} --
  this is equivalent since $n^{(1)}$ solves the  ZK
  equation \eqref{newZK}.
\end{itemize}

\medbreak

Finally we  get the following consistency result that generalized
Proposition \ref{propconsist} when isothermal pressure is taken into account.
\begin{proposition}\label{propconsistbis}
Let $T_0>0$, $n_0\in H^s(\R^d)$ ($s\geq 5$) and $n^{(1)}\in C(\lbrack 0,
T_0\rbrack;H^s(\R^d))$, solving (with $c=\sqrt{1+\alpha}$),
$$2c\partial_T
n^{(1)}+2cn^{(1)}\partial_Xn^{(1)}+\left(1+\frac{c}{a^2}\right)\Delta \partial_Xn^{(1)}=0,\qquad n^{(1)}_{\vert_{t=0}}=n_0.$$
Constructing the other profiles as indicated above, the approximate
solution 
 $(n^\eps,\bv^\eps,\phi^\eps)$ given by \eqref{ZKansatzbis} solves
 (\ref{scaEPgen}) up to order $\eps^3$ in $\phi^\eps$, $\eps^2$ in $n^\eps$, $v_x^\eps$,
 and up to order $\eps^{3/2}$ in $v_y^\eps,$ $v_z^\eps$:
\begin{equation}
    \label{scaEPquar}
    \left\lbrace
   \begin{array}{lcl}
   	\dsp\partial_tn^\e+\nabla\cdot ((1+\eps n^\e){\bf v}^\e)
        =\e^{2} N^\e,\\
	\dsp\partial_t{\bf v}^\e+\eps ({\bf v}^\e\cdot \nabla){\bf
          v}^\e+\nabla\phi^\e+\alpha\frac{\nabla n}{1+\eps n}+a\eps^{-1/2}{\bf e}\wedge {\bf
          v}^\e=\eps^{3/2} R^\e,\\
	\dsp -\epsilon
        ^2\Delta\phi^\e+e^{\epsilon\phi^\e}-1=\eps n^\e+ \epsilon^3 r^\e,
    \end{array}\right.
\end{equation}
with $R^\eps=(\eps^{1/2} R^\e_1,R^\e_2,R^\e_3)$ and
$$
\abs{N^\e}_{L^\infty([0,\frac{T_0}{\eps}];H^{s-5})} +\sum_{j=1}^3\abs{R^\e_j}_{L^\infty([0,\frac{T_0}{\eps}];H^{s-5})}  +\abs{r^\eps}_{L^\infty([0,\frac{T_0}{\eps}];H^{s-4})}\leq C\big(T_0,\abs{n_0}_{H^s}\big).
$$
\end{proposition}

\subsection{Justification of the  Zakharov-Kuznetsov approximation}

Proceeding exactly as for Theorem \ref{thmjustif} but replacing
Theorem \ref{thmIVPunif} by Theorem \ref{thmIVPunifgen}, we get the
following justification of the Zakharov-Kuznetsov approximation in presence of a isothermal pressure.
\begin{theorem}\label{thmjustifgen}
Let $n^0\in H^{s+5}$, with $s>d/2+1$, such that $1+\eps n^0\geq c_0$
on $\R^d$ for some constant $c_0>0$. There exists $T_1>0$ such that
\item{\bf i.} The  Zakharov-Kuznetsov approximation $(n^\e,\bv^\e,\phi^\e)$ of Proposition
  \ref{propconsistbis} exists on the time interval $[0,T_1/\eps]$;
\item{\bf ii.} There exists a unique solution $(\underline{n},\underline{\bf v},\underline{\phi})\in C(\lbrack
 0,\frac{{T_1}}{\eps}\rbrack; H^{s+1}_{\eps\alpha}(\R^d)\times
 H^{s+1}_\eps(\R^d)^d\times H^{s+1}(\R^d))$ provided by Theorem
 \ref{thmIVPunifgen} to the Euler-Poisson equations with isothermal pressure \eqref{scaEPgen} with
 initial condition $(\underline{n}^0,\underline{\bf
   v}^0,\underline{\phi}^0)=(n^\e,\bv^\e,\phi^\e)_{\vert_{t=0}}$.
\item Moreover, one has the error estimate
$$
\forall 0\leq t\leq T_1/\eps, \qquad \abs{\underline{n}(t)-n^\e(t)}_{H^{s+1}_{\eps\alpha}}^2+\abs{
  \underline{\bv}(t)-\bv^\e(t)}_{H^{s+1}_\eps}^2\leq \eps^{3/2} t C(\frac{1}{c_0},T_1,\abs{n^0}_{H^{s+5}}).
$$
\end{theorem}
\begin{remark}
The comments made in Remarks \ref{rmm1} and \ref{rmm2} on the precision of the
Zakharov-Kuznetsov approximation for cold plasmas can be transposed to
the more general case considered here.
\end{remark}
\begin{remark}
As already said, the ZK equation \eqref{newZK} coincides in
dimension $d=1$ with the  KdV equation derived in
\cite{GuoPu}. A consequence of the uniformity of the existence time
with respect to $\alpha$ in Theorem \ref{thmIVPunifgen} is that
Theorem \ref{thmjustifgen} provides a justification on a time scale of
order $O(1/\eps)$ which is uniform with respect to $\alpha$ whereas it
shrinks to zero when $\alpha\to 0$ in \cite{GuoPu}. 
\end{remark}
\begin{merci}
The three authors acknowledge the support of IMPA,  the Brasilian-French program in Mathematics and the MathAmSud Project "Propagation of Nonlinear Dispersive Equations". D. L. acknowledges support from the project  ANR-08-BLAN-0301-01 and J.-C. S. from the project ANR-07-BLAN-0250 of the Agence Nationale de la Recherche.
\end{merci}

\end{document}